\documentclass{amsart}

\usepackage{amsfonts,amssymb,amscd,amsmath,enumerate,verbatim,calc}
\usepackage{euscript}
\usepackage[all]{xy}

\newcommand{\CM}{Cohen-Macaulay}
\newcommand{\fg}{finitely generated \ }

\newcommand{\eF}{\EuScript{F}}
\newcommand{\ff}{\text{if and only if}}
\newcommand{\wrt}{with respect to}
\newcommand{\f}{\hat{f} }

\newcommand{\n}{\mathfrak{n} }
\newcommand{\m}{\mathfrak{m} }
\newcommand{\tf}{\mathfrak{t} }
\newcommand{\M}{\mathfrak{M} }
\newcommand{\N}{\mathfrak{N} }
\newcommand{\q}{\mathfrak{q} }

\newcommand{\A}{\mathfrak{a} }
\newcommand{\B}{\mathfrak{b} }
\newcommand{\C}{\mathfrak{c} }
\newcommand{\R}{\mathcal{R} }
\newcommand{\Pc}{\mathcal{P} }

\newcommand{\Z}{\mathbb{Z} }

\newcommand{\nZ}{n \in \mathbb{Z} }

\newcommand{\spr}{\mathfrak{s} }
\newcommand{\GA}{G_{\mathfrak{a}}(A) }

\newcommand{\GB}{G_{\mathfrak{b}}(B) }
\newcommand{\GT}{G_{\mathfrak{t}}(T) }
\newcommand{\ra}{\EuScript{R}_{\mathfrak{a}}(A) }
\newcommand{\GF}{G_{\EuScript{F}}(A) }

\newcommand{\tF}{\EuScript{S}_{\EuScript{F}}(A)}
\newcommand{\ral}{\EuScript{R}_{\mathfrak{a}^l}(A) }

\newcommand{\ta}{\EuScript{S}_{\mathfrak{a}}(A) }
\newcommand{\tM}{\EuScript{S}_{\mathfrak{a}}(M) }

\newcommand{\eR}{\EuScript{R}}
\newcommand{\xb}{\mathbf{x}}
\newcommand{\ub}{\mathbf{u}}

\newcommand{\eG}{\EuScript{G}}

\newcommand{\eH}{\EuScript{H}}
\newcommand{\Sc}{\mathcal{S} }
\newcommand{\rt}{\rightarrow}
\newcommand{\xar}{\longrightarrow}
\newcommand{\ov}{\overline}
\newcommand{\sub}{\subseteq}

\newcommand{\bX}{\mathbf{X} }
\newcommand{\bY}{\mathbf{Y} }

\newcommand{\Bcal}{\mathcal{B} }

\newcommand{\dd}{*}

\newcommand{\wt}{\widetilde }

\newcommand{\grade}{\operatorname{grade}}
\newcommand{\depth}{\operatorname{depth}}
\newcommand{\ann}{\operatorname{ann}}
\newcommand{\red}{\operatorname{red}}
\newcommand{\htt}{\operatorname{ht}}

\newcommand{\Hs}{\operatorname{\ ^* Hom}}
\newcommand{\Es}{\operatorname{\ ^* Ext}}

\newcommand{\Hom}{\operatorname{Hom}}

\newcommand{\Ext}{\operatorname{Ext}}

\theoremstyle{plain}

\newtheorem{thm}{Theorem}

\newtheorem{theorem}{Theorem}[section]
\newtheorem{corollary}[theorem]{Corollary}
\newtheorem{lemma}[theorem]{Lemma}
\newtheorem{proposition}[theorem]{Proposition}

\theoremstyle{definition}
\newtheorem{definition}[theorem]{Definition}
\newtheorem{defn}[thm]{Definition}
\newtheorem{s}[theorem]{}
\newtheorem{remark}[theorem]{Remark}
\newtheorem{remarkC}[thm]{Remark}
\newtheorem{example}[theorem]{Example}

\theoremstyle{remark}
\newtheorem{observation}[theorem]{Observation}

\numberwithin{equation}{theorem}

\begin{document}

\title{Itoh's conjecture for normal ideals}
\author{Tony~J.~Puthenpurakal}
\date{\today}
\address{Department of Mathematics, Indian Institute of Technology, Bombay, Powai, Mumbai 400 076, India}

\email{tputhen@math.iitb.ac.in}
\subjclass{Primary  13A30,  13D45 ; Secondary 13H10, 13H15}
\keywords{multiplicity,  reduction, Hilbert polynomial, associated graded rings}
  \begin{abstract}
Let $(A,\m)$ be an analytically unramified  \CM \ local ring and let $\A$ be an $\m$-primary ideal in $A$. If $I$ is an ideal in $A$ then let $I^*$ be the integral closure of $I$ in $A$. Let $\GA^*  = \bigoplus_{n\geq 0 }(\A^n)^*/(\A^{n+1})^*$ be the associated graded ring of the integral closure filtration of $\A$. Itoh conjectured  in 1992 that if $e_3^{\A^*}(A) = 0$ and $A$ is Gorenstein then $\GA^*$ is \CM. In this paper we prove an important case of Itoh's conjecture: we show that if $A$ is \CM \  and if  $\A$ is normal (i.e., $\A^n$ is integrally closed for all $n \geq 1$)
with  $e_3^\A(A) = 0$ then $G_\A(A)$ is \CM.
  \end{abstract}
 \maketitle

\section{Introduction}
\begin{remark}
This paper consists of part of the author's paper \cite{PuCI}. This was done due to advice of  some of my colleagues. The other parts of
  \cite{PuCI} will be published later in a separate paper.
\end{remark}
Let $(A,\m)$ be a Noetherian   local ring of dimension $d$ and let $\A$ be  an $\m$-primary ideal in $A$. We consider multiplicative $\A$-stable filtration's i.e.,
$\eF = \{ \A_n \}_{n\geq 0}$ such that
\begin{enumerate}[\rm (i)]
  \item $\A_n$ are ideals in $A$ for all $n \geq 0$ with  $\A_0 = A$.
  \item $\A \subseteq \A_1$ and $\A_1 \neq A$.
  \item $\A_n \A_m \subseteq \A_{n + m}$ for all $n, m \geq 0$.
  \item $\A \A_n = \A_{n + 1}$ for all $n \gg 0$.
\end{enumerate}
Let $\ta = \bigoplus_{n\geq 0} \A^n$ be the Rees ring of $A$ \wrt \ $\A$. Let $\tF =  \bigoplus_{n\geq 0} \A_n$ be the Rees ring of $A$ \wrt \ $\eF$. We have an inclusion of rings $\ta \hookrightarrow \tF$ and $\tF$ is a finite $\ta$-module. We  give some examples of multiplicative $\ta$-stable filtration's:

$(0)$: Of course the first example of $\ta$-stable filtration is the $\A$-adic filtration $\{ \A^n \}_{n \geq 0}$.

$(1)$ Assume $\depth A > 0$. Consider
$$\widetilde{\A^n} = \bigcup_{r \geq 1} (\A^{n+r} \colon \A^r).$$
Then $\{ \widetilde{\A^n} \}_{n \geq 0}$ is the well-known Ratliff-Rush filtration of $A$ \wrt \ $\A$. It is well-known that $\widetilde{\A^n} = \A^n$ for all $n\gg 0$. Furthermore it is easily verified that $\{ \widetilde{\A^n} \}_{n \geq 0}$ is an $\A$-stable filtration, see \cite{RR}.

$(2)$ Assume $A$ is analytically unramified, i.e., the completion of $A$ is reduced. In this case there are two $\A$-stable filtration which is of interest:
\begin{enumerate}[\rm (i)]
\item
If $I$ is an ideal in $A$ then by $I^*$ we denote the integral closure of $I$ in $A$.
The integral closure filtration of  $\A$ is $\{ (\A^n)^* \}_{n \geq 0}$.  As $A$ is analytically unramified the integral closure filtration of $\A$ is $\A$-stable,  see \cite{Rees} (also see \cite[9.1.2]{Hun-Sw}).  The integral closure filtration of $\m$-primary ideals  has inspired plenty of research, see for instance \cite[Appendix 5]{ZS}, \cite{Rees-2}, \cite{Hun}, \cite{It}, \cite{ItN}, \cite{Masuti-1}, \cite{Masuti-2}
\cite{G}, \cite{GHM}, \cite{CPR} and \cite{Ver}.
\item
If $A$ contains a field of characteristic $p >0$ then the tight closure filtration of $\A$ has been  of some  interest recently,  see \cite{GVM}.
\end{enumerate}

Let $\GA = \bigoplus_{n \geq 0} \A^n/\A^{n+1}$ be the associated graded ring of $A$ \wrt \ $\A$. Let $\eF = \{ \A_n \}_{n \geq 0}$ be an $\A$-stable filtration. Consider
$\GF = \bigoplus_{n \geq 0} \A_n/\A_{n+1}$ be the associated graded ring of $A$ \wrt \ $\eF$. Let $\ell(M)$ denote the length of an $A$-module $M$.  Clearly $\GF $ is a finite $\GA$-module. It follows the function
$H_{\eF}(n) = \ell(A/\A_{n +1})$  is of polynomial type, i.e., there exists a polynomial $P_\eF(X) \in \mathbb{Q}[X]$ of degree $d$ such that $P_{\eF}(n) = H_{\eF}(n)$ for all $n \gg 0$. We write
$$P_\eF(X) = e_0^\eF(A)\binom{X+d}{d} -  e_1^\eF(A)\binom{X+d -1}{d -1} + \cdots + (-1)^d e_d^\eF(A). $$
If $\eF$ is the $\A$-adic filtration then we set $e_i^\A(-)= e^\eF_i(-)$. Also if $\eF$ is the $\A$-integral closure filtration then we set ${e_i^\A}^*(-)= e^\eF_i(-)$.
Clearly $e_0^\eF(A) = e_0^\A(A) > 0$, for instance see \cite[11.4]{AM}.

\emph{For the rest of the paper we will assume that $A$ is a \CM \ local ring with residue field $k$.}\\
Northcott showed that $e_1^\A(A) \geq 0$. Furthermore if $k$ is infinite then $e_1^\A(A) = 0$ if and only if $\A$ is a parameter ideal, see \cite{North}.
It is well-known  that for any $\A$-stable  multiplicative filtration we have $e_1^\eF(A) \geq 0$.
Narita showed that $e_2^\A(A) \geq 0$.  Furthermore if $k$ is infinite and $d = 2$ then $e_2^\A(A) = 0$ if and only if reduction number of $\A^n$ is one for all $n \gg 0$, see \cite{Nar}. Narita also gave an example of a ring with $e_3^\A(A) < 0$.
We now consider the integral closure filtration of $\A$. In this case ${e_2^\A}^*(A) \geq 0$. Itoh proved that ${e_3^\A}^*(A) \geq 0$, see \cite[Theorem 3]{ItN}. He conjectured that if $A$ is Gorenstein and
${e_3^\A}^*(A) = 0$ then $\GA^* = \bigoplus_{n \geq 0} (\A^n)^*/(\A^{n+1})^*$ is \CM. He proved the result when
${\A}^* = \m$, see \cite[Theorem 3]{ItN} (also see \cite[3.3]{CPR}). See \cite{KM} for a geometric view-point of Itoh's conjecture.

Recall an ideal $I$ is normal if $I^n$ is integrally closed for all $n \geq 1$. In this paper we solve Itoh's conjecture for normal $\m$-primary ideals. Note if $A$ contains an
$\m$-primary normal ideal then it is automatically analytically unramified, (cf., \cite[p.\ 153]{HM} and \cite[9.1.1]{Hun-Sw}). We prove:
\begin{theorem}\label{main-new}
Let $(A,\m)$ be a \CM \ local ring and let $\A$ be an $\m$-primary normal ideal with $e_3^\A(A) = 0$. Then $\GA$ is \CM.
\end{theorem}
\emph{Note we are not assuming $A$ is Gorenstein in Theorem \ref{main-new}.}

The proof of Theorem \ref{main-new} requires three techniques, two of which is introduced in this paper  and one which was introduced in an earlier paper of the author. We describe the three techniques.

\emph{First Technique:}\\
Let $\mathcal{P}$ be a property of  Noetherian rings. For instance $\Pc =$ regular, complete intersection (CI), Gorenstein, \CM \ etc.
\begin{defn}
Let $(A,\m)$ be a Noetherian local ring and let $\A$ be a proper ideal
in $A$. We say $A$ admits a \emph{$\Pc$-approximation} \wrt \ $\A$ if there exists
a  local ring $(B,\n)$, an ideal $\B$ of $B$  and $\psi \colon B \rt A$,  a local ring homomorphism
with $\psi(\B)A =  \A$ such that the following \emph{three}
properties hold:
 \begin{enumerate}
\item
$A$ is a finitely generated $B$ module ( via $\psi$).
\item
$\dim A = \dim B $.
\item
$B$ and $\GB$  have property $\Pc$.
\end{enumerate}
 If the above conditions hold we say $[B,\n,\B,\psi]$ is a \emph{$\Pc$-approximation}
of $[A,\m,\A]$.
\end{defn}

\begin{remarkC}

\noindent $\bullet$  Regular approximations \textit{seems} to be  rare while Cohen-Macaulay approximations \emph{don't seem} to  have many applications.

  \noindent $\bullet$   For applications we will often insist that $\psi$, $\B$ satisfy some additional properties.
\end{remarkC}

When $A$ is complete we prove the following general result(see \ref{equiTT}).  Let $\spr(\A)$ denote the analytic spread of $\A$, see \ref{equimultiple-defn}.

\begin{thm}\label{introGAPP}
Let $(A,\m)$ be  a complete  Noetherian local ring and let $\A$ be a proper ideal
 in $A$ such that $ \spr(\A) + \dim A/\A = \dim A$. Then
$A$  admits a  CI-approximation.
\wrt \ $\A$.
\end{thm}

Notice the hypothesis of Theorem \ref{introGAPP} is satisfied  if $\A$ is  $\m$-primary.
Another significant case when Theorem \ref{introGAPP} holds is when $A$ is equidimensional and $\A$ is equimultiple; see \cite[2.6]{HUO}.  For definition of an equimultiple ideal see \ref{equimultiple-defn}. For an indication why this technique is useful, see Remark \ref{pulp}.

\emph{Second Technique:}\\
Before we state our second step we need some preliminaries.
Let $M$ be a finitely generated $A$-module.
Recall an $\A$-\textit{filtration} $\eG = {\{M_n\}}_{n \in \Z}$ on $M$ is a
collection of submodules of $M$ with the properties
\begin{enumerate}
\item
$M_n \supset M_{n+1}$ for all $n \in \Z$.
\item
$M_n = M$ for all $n \ll 0$.
\item
$\A M_n \subseteq M_{n+1}$ for all $n \in \Z$.
\end{enumerate}
If
$\A M_n = M_{n+1}$ for all $n \gg 0$ then we say $\eG$ is $\A$-\emph{stable}.

Let $\ra = \bigoplus_{n\in \Z} \A^n$ be the extended Rees algebra of $A$ \wrt \ $A$.
If $\eG = \{M_n\}_{n \in \Z}$ is an $\A$-stable filtration on $M$ then  set
$\eR(\eG,M) = \bigoplus_{n \in \Z}M_nt^n$ the \emph{extended Rees-module} of $M$
\emph{\wrt }\ $\eG$.
 Notice that $\eR(\eG,M)$ is a finitely generated graded $\ra$-module. Let $G(\eG, M) = \bigoplus_{n\in \Z}M_n/M_{n+1}$ be the associated graded module of $M$ \wrt \ $\eG$. Note
 $G(\eG,M)$ is a finitely generated graded $G_\A(A)$-module.

Set $M^* = \Hom_A(M, A)$. Set $M^* = \Hom_A(M, A)$. For $n \in \Z$ set
$$  M_n^* = \{ f \in M^* \mid    f(M) \subseteq \A^n \} \cong \Hom_A(M, \A^n).      $$
It is well-known that $M_{n+1}^* = \A M_n^*$ for all $n \gg 0$.
Note $M_n^* = M^*$ for $n \leq 0$. Set $\eF = \{M_n^*\}_{n \in \Z}$. We call this filtration \emph{ the dual filtration } of $M^*$ \wrt \  $\A$. This filtration is
classical cf.  \cite[p.\ 12]{Ser}.  However it has not been used before in the study of blow-up algebra's (modules) of \CM \ rings (modules).
 \emph{We construct  an explicit isomorphism (see \ref{main}):}
$$\Psi_M \colon \eR(\eF,M^{*}) \quad \xar \Hs_\R(\eR(\A, M),\ra ). $$
It is an important observation that dual filtration's behaves well  \wrt \ the Veronese functor, see \ref{dual-Ver-prop}.
\begin{remark}\label{pulp}
Let $(A,\m)$ be a complete \CM \ local ring of dimension $d$ and let $\A$ be an $\m$-primary ideal. Let $[B,\n,\B,\psi]$ be a \emph{CI-approximation}
of $[A,\m,\A]$. Then note $\eR(\B, A) = \eR(\A, A)$.  Note $\omega_A = \Hom_B(A, B)$ is a canonical module of $A$. The dual filtration $\eF$ on $\omega_A$ is isomorphic
to $\Hs_\R(\ra,\R_\B(B))$. As $\R_\B(B)$ is Gorenstein the dual of later module up-to shift is the top local cohomology of $\ra$ \wrt \ to its $*$-maximal ideal.
To prove our result we have to work on the top local cohomology of $\ra$ whose dual up-to shift is an $\A$-stable filtration on $\omega_A$ which is \CM \ and so it is tractable to techniques already developed by several authors.
\end{remark}

\emph{Third Step:}\\
In \cite{Pu5} we introduced $L^\A(M) = \bigoplus_{n \geq 0}M/\A^{n+1}M$ where $M$ is a \CM \ $A$-module. Note $\ta$ is a subring of $A[t]$ which has proved useful to tackle many problems.
Let  $\tM = \bigoplus_{n \geq 0} \A^nM$ be the Rees module of $M$ \wrt \ $\A$. Note $\tM $ is a finitely generated $\ta$-module.
We have an exact sequence of $\ta$-modules
\[
0 \rt \tM \rt M[t]    \rt  L^\A(M)[-1]   \rt 0. \tag{$\dagger$}
\]
This yields a $\ta$-module structure on $L^\A(M)(-1)$ and so on $L^\A(M)$.  Notice $L^\A(M)$ is not finitely generated as a $\ta$-module.
Let $\M = (\m, \ta_+)$ be the $*$-maximal ideal of $\ta$. Assume $r = \dim M \geq 1$. In \cite{Pu5} we proved that for $i = 0, \cdots, r-1$ the local cohomology modules $H^i_\M(L^\A(M))$ satisfy
some nice properties, see section \ref{Lprop} for a summary of these properties. Note that using $(\dagger)$ we get an inclusion of graded $\ta$-modules
$$ 0 \rt H^{r-1}_\M(L^I(M))[-1] \rt H^r_\M (\tM)$$
Note it is possible that $H^r_\M (\tM)_n $ can have infinite length for some $n$. However $H^{r-1}_\M(L^I(M))_n $ has finite length for all $n \in \Z$.

An ideal $I$ is said to be asymptotically normal if $I^n$ is integrally closed for all $n \gg 0$. Itoh's conjecture for normal ideals easily follows from the following result:
\begin{lemma}\label{crucial-itoh-2-introduction}
Let $(A,\m)$ be a \CM \ local ring of dimension $d \geq 3$ and let $\A$ be an asymptotically  normal $\m$-primary ideal with $e_3^\A(A) = 0$.  Then the local cohomology modules $H^i(L^\A(A))$ vanish for $1 \leq i \leq d -1$.
\end{lemma}

Here is an overview of the contents of this paper. In section two we introduce notation and discuss a few
preliminary facts that we need. In section 3 we prove the result on complete intersection approximations  when $\A = \m$ and $A$ is a quotient of a regular local ring.
We also prove it when
$A$ is complete, $\A$ is $\m$-primary  and contains a field.
In section four we discuss dual filtrations
In section five we prove our result on complete intersection approximation of equimultiple ideals.
In section six we discuss some preliminaries on $L^\A(A)$ that we need to prove Itoh's conjecture.
Finally in our last section we prove Itoh's conjecture for normal ideals.

\section{preliminaries}

In this paper all rings are commutative Noetherian and all modules  (unless stated otherwise) are assumed to
be finitely generated. Let $A$ be a local ring, $\A$ an ideal in $A$ and let $N$ be an $A$-module. Then set
 $\ell(N)$ to be length of $N$ and $\mu(N)$ the number of minimal generators of $N$.

\s Let $(A,\m)$ be a    local ring of dimension $d$,
 $\A$  an $\m$-primary ideal in $A$ and let $M$ be an $A$-module   of dimension $r$. Set $G_{\A}(A) = \bigoplus_{n \ge 0} \A^n/\A^{n+1} $;  the
\textit{associated graded ring} of $A$ \wrt \ $\A$   and let $G_\A(M)= \bigoplus_{n \ge 0} \A^nM/\A^{n+1} M$ be the \textit{associated graded module} of $M$ \wrt \ $\A$.

The Hilbert-Samuel function of $M$ with respect to $\A$ is the function
\[
n \mapsto \ell(M/\A^{n+1}M)\quad \text{for all} \ n \geq 0.
\]

It is well known that for large values of $n$ it is given by a polynomial $P^{\A}_{M}(n)$ of
degree $r = \dim M$, called the Hilbert-Samuel polynomial of $M$ with respect to $\A$.
It can be written in the form
\[
P^{\A}_{M}(X) = \sum_{i=0}^{r}(-1)^ie_{i}^{\A}(M)\binom{X + r - i}{r-i}
\]

The integers $e_{0}^{\A}(M), e_{1}^{\A}(M), ..., e_{r}^{\A}(M)$ are
called the \emph{Hilbert coefficients} of $M$ with
respect to $\A$. The number $e_{0}^{\A}(M)$ is also called the \emph{multiplicity} of $M$ with respect to $\A$. The existence of the Hilbert-Samuel polynomial is equivalent to the fact that
the formal power series $\sum_{n \geq 0}\ell(M/\A^{n+1}M)z^n$
represents a rational function of a
special type:
\begin{align*}
\sum_{n \geq 0}\ell(M/\A^{n+1}M)t^n &= \frac{h^{\A}_{M}(t)}{(1-t)^{r+1}}\quad \text{where}
\ r = \dim M \ \text{and}\\
h_{M}^{\A}(t) &= h_{0}^{\A}(M) + h_{1}^{\A}(M)t + \cdots + h_{s}^{\A}(M)t^s \in
\mathbb{Z}[t]
\end{align*}

\s\label{hilb-gt-d} If $f$ is a polynomial the we set $f^{(i)}$ to denote the $i^{th}$ formal derivative of $f$. It can be easily verified that
$$ e_{i}^{\A}(M) = \frac{{h_{M}^\A}^{(i)}(1)}{i!} $$
for $i = 0, \ldots,r$.  It is also convenient to set
$$e_{i}^{\A}(M) = \frac{{h_{M}^\A}^{(i)}(1)}{i!}  \quad \text{for all} \ i \geq 0.$$

\s \label{mod-N}
Let $\eG = {\{M_n\}}_{n \in \Z}$ be an $\A$-filtration on $M$ and let $N$ be a submodule of $M$. By the \emph{quotient filtration} on $M/N$ we mean the filtration   $\ov{\eG} = \{ (M_n + N)/N \}_{\nZ}$.
If $\eG$ is an $\A$-stable filtration on
$M$ then $\ov{\eG}$ is an $\A$-stable filtration on
$M/N$.\textit{ Usually} for us $N =\xb M$ for some $\xb = x_1,\ldots,x_s \in \A$.

\s If $\eG = \{M_n\}_{n \in \Z}$ is an $\A$ stable filtration on $M$ then  set
$\eR(\eG,M) = \bigoplus_{n \in \Z}M_nt^n$ the \emph{extended Rees-module} of $M$
\emph{\wrt }\ $\eG$.
 Notice that $\eR(\eG,M)$ is a finitely generated graded $\ra$-module. If $\eG$ is the usual $\A$-adic filtration then
set \\  $ \eR(\A, M) = \eR(\eG,M)$.

Set $G(\eG,M)=\bigoplus_{n \in \Z} M_n/M_{n+1}$, the \emph{associated graded
module }of $M$ \emph{\wrt} \ $\eG$. Notice $G(\eG,M)$ is a finitely generated graded module over
 $G_{\A}(A)$.
 Furthermore $\eR(\eG,M) /t^{-1}\eR(\eG,M) = G(\eG, M)$.

\s \label{shift-filt}
Let $\eG = {\{M_n\}}_{n \in \Z}$ be an $\A$-filtration on $M$ and let $s \in \Z$. By the $s$-\emph{th shift} of $\eG$, denoted by $\eG(s)$ we mean the filtration ${\{ \eG(s)_n\}}_{n \in \Z}$ where
$\eG(s)_n = \eG_{n+s}$. Clearly
$$G(\eG(s), M) = G(\eG, M)(s) \quad \text{and} \quad \R(\eG(s), M) = \R(\eG, M)(s). $$

\s All   filtration's in this paper $\eG = \{M_n\}_{n \in \Z}$ will be \textit{separated} i.e.,
 $\bigcap_{\nZ}M_n = \{0\}$.
This is automatic if $A$ is local,  $\A \neq A$ and $\eG$ is $\A$-stable.
If $m$ is a non-zero element of $M$ and if $j$ is the largest integer such that
$ m \in M_j$,
then we let $m^{*}_{\eG}$ denote the image of $m $ in $M_j  \setminus M_{j+ 1}$ and we call it the \textit{initial form} of $m$ \wrt \ $\eG$. If $\eG$ is clear from the context then we drop the subscript $\eG$.

\s \emph{Veronese functor:}. Let $\eF = \{M_n \}_{n\in \Z}$ be an $\A$-stable filtration. For $l \geq 1$ set $\eF^{<l>}= \{M_{nl} \}_{n \in \Z}$. Notice $\eF^{<l>}$ is an $\A^l$-stable filtration. We also have an $\ral$-isomorphism $\eR(\eF,M)^{<l>} \cong \eR(\eF^{<l>}, M)$.

\s For definition and  basic properties of superficial sequences see \cite[p.\ 86-87]{Pu1}.

\s\label{AtoA'} \textbf{Flat Base Change:} In our paper we do many flat changes of rings.
 The general set up we consider is
as follows:

 Let $\phi \colon (A,\m) \rt (A',\m')$ be a flat local ring homomorphism
 with $\m A' = \m'$. Set $\A' = \A A'$ and if
 $N$ is an $A$-module set $N' = N\otimes_A A'$. Set $k = A/\m$ and $k' = A'/\m'$.

 \textbf{Properties preserved during our flat base-changes:}

\begin{enumerate}[\rm (1)]
\item
$\ell_A(N) = \ell_{A'}(N')$. So
 $H_{\A}(M,n) = H_{\A'}(M',n)$ for all $n \geq 0$.
\item
$\dim M = \dim M'$ and  $\grade(K,M) = \grade(KA',M')$ for any ideal $K$ of $A$.
\item
$\depth G_{\A}(M) = \depth G_{\A'}(M')$.
\end{enumerate}

\textbf{Specific flat Base-changes:}

\begin{enumerate}[\rm (a)]
\item
$A' = A[X]_S$ where $S =  A[X]\setminus \m A[X]$.
The maximal ideal of $A'$ is $\n = \m A'$.
The residue
field of $A'$ is $k' = k(X)$. Notice that $k'$ is infinite.
\item
$A' =  \widehat{A}$  the completion of $A$ \wrt \ the maximal ideal.
\item
Applying (a) and (b) successively ensures that $A'$ is complete with $k'$ infinite.
\item
$A' = A[X_1,\ldots,X_n]_S$ where $S =  A[X_1,\ldots,X_n]\setminus \m A[X_1,\ldots,X_n]$.
The maximal ideal of $A'$ is $\n = \m A'$.
The residue
field of $A'$ is $l = k(X_1,\ldots,X_n)$. Notice that if $I$ is integrally closed then $I'$ is
also integrally closed. Recall an ideal $K$ is said to be asymptotically normal if  $K^n$ is integrally closed for all $n \gg 0$.
When $\dim A \geq 2$, $A$ is  \CM \ and $I$ is asymptotically normal, say $I = (a_1,\cdots,a_n)$ then in \cite[Corollary 2]{Ciu} it is proved that
the element $y = \sum_{i = 1}^{n} a_iX_i$ in $A'$ is $I'$-superficial and
the $A'/(y)$ ideal $J = I'/(y)$ is asymptotically normal. We call $A'$ a general extension of $A$.
\end{enumerate}

\s \label{lying-above} We will need the following result. Let $\phi \colon (B,\n) \rt (A.\m)$ be a local map such that $A$ is finitely generated $B$-module (via $\phi$). Consider
$R = B[X_1,\cdots,X_n]$ and $S = A[X_1,\cdots. X_n]$ and let $\psi \colon R \rt S$ be the map induced by $\phi$. Note $S$ is a finite $R$-module via $\psi$.
Let $P = \n R$ and $Q = \m S$. Then $S_P = S_Q$. To see this note that $Q$ is the only prime in $S$ lying over $P$. The result follows from an exercise problem
in \cite[Exercise 9.1]{Ma}.

\s\label{equimultiple-defn}  The \emph{fiber cone} of $\A$ is the
$k$-algebra
$F(\A) = \bigoplus_{n\geq 0}\A^n/\m \A^{n}.$
Set $\spr(\A) = \dim F(\A)$,
the \emph{analytic spread} of $\A$.
 Recall $\dim A \geq \spr(\A) \geq \htt \A$.
We say $\A$ is\textit{ equimultiple} if
$\spr(\A) = \htt \A$. Clearly $\m$-primary ideals are equimultiple.

\s \label{redNO} We assume $k = A/\m$ is infinite. Let $\C = (x_1,\ldots,x_l)$ be  a minimal reduction of
$\A$ \wrt \ $M$.  We denote by  $\red_\C(\A, M) := \min \{ n  \mid \  \C \A^n M= \A^{n+1}M \} $
the \textit{reduction number} of $\A$ \wrt \ $\C$ and $M$. Let
 $$\red(\A, M) = \min \{ \red_\C(\A, M) \mid \C
\text{\ is  a minimal reduction of \ }  \A \}$$
 be the \textit{reduction number } of $M$ \wrt \ $\A$. Set
$\red_\A(A) = \red(\A, A)$.

\s \label{gen-cm} Let $R= \bigoplus_{n\geq 0}R_n$ be a standard graded algebra over an Artin local ring $(R_0,\m_0)$. Set $\M = \m_0 \oplus R_+$. We say that $R$ is generalized \CM \ if $H^i_{\M}(R)$ has finite length for $i < \dim R$.

\section{CI-Approximation: Some special cases}\label{Gorapp-special}
In this section we prove that if $(A,\m)$ is a quotient of a regular local ring then
$[A,\m,\m]$ has a CI-approximation $[B,\n,\n, \phi]$ with $\phi$ onto. In Theorem \ref{equi} we  state our main result regarding CI-approximations. This will be proved in section \ref{sectionGAPP} . Finally in Theorem \ref{field} we prove a result regarding CI-approximation of an $\m$-primary ideal $\A$ in a complete  equicharacteristic
local ring $(A,\m)$.
The following Lemma  regarding annihilators is crucial.

\begin{lemma}
\label{annihilator}
Let $M$ be an $A$-module and let $\A$ be an ideal. Set $G = \GA$.
 Suppose there exists
 $\xi_1,\ldots,\xi_s \in \A$ such that $\xi_{i}^{*} \in \ann_{G} G_{\A}(M)$
for each $i$. Also assume that
 $\xi_{1}^{*},\ldots,\xi_{s}^{*}$ is a $G$-regular sequence.
 Then there exists $u_1,\ldots,u_s \in \A$ such that
\begin{enumerate}[\rm (1)]
\item
$u_1,\ldots,u_s$ is an $A$-regular sequence.
\item
$u_i \in \ann M$ for each $i = 1,\ldots,s$.
\item
For each $i \in \{1,\ldots,s \}$ there exists $n_i \geq 1$ such
 that $u_{i}^{*} = (\xi_{i}^{*})^{n_i}$.
\item
 $u_{1}^{*},\ldots,u_{s}^{*}  \in  \ann_G G_\A(M).$
\item
 $u_{1}^{*},\ldots,u_{s}^{*} $ is a  $G$-regular sequence.
\end{enumerate}
\end{lemma}
\begin{proof}
Suppose we have constructed $u_1,\ldots,u_s$ satisfying (2) and (3) then
 $u_1,\ldots,u_s$ satisfy all the remaining properties. This can be seen as follows:

(4) follows from (3); since $\xi_{i}^{*} \in \ann_{G} G_{\A}(M)$ for each $i$.

(1) and (5). As  $\xi_{1}^{*},\ldots,\xi_{s}^{*}$ is a $G$-regular sequence
 we also get
$(\xi_{1}^{*})^{n_1},\ldots,(\xi_{s}^{*})^{n_s}$ is a $G$-regular sequence
\cite[16.1]{Ma}.
Thus $u_{1}^{*},\ldots,u_{s}^{*} $ is a  $G$-regular sequence. It follows from   \cite[2.3]{VV} that
$u_1,\ldots,u_s$ is an $A$-regular sequence.
Thus it suffices to show there exists $u_1,\ldots,u_s$ satisfying (2) and (3).

Fix $i \in \{1,\ldots,s \}$.
Set $\xi = \xi_i$.
  Say $\xi \in \A^r \setminus \A^{r+1}$ for some $r \geq 1$.
Since $\xi^* \in \ann_G G_{\A}(M)$ we have $\xi M \subseteq \A^{r+1}M$.
Set $\q = \A^{r+1}$.
By the \emph{determinant trick}, \cite[2.1]{Ma}, there exists a monic polynomial $f(X) \in A[X]$ such that
\begin{align*}
f(X) &= X^{n} + a_1X^{n-1} + \ldots + a_{n-1}X + a_n \quad \text{with} \quad a_i \in \q^i, \text{\ for} \ i = 1,\ldots n \\
\text{and} \ u = f(\xi) &= \xi^{n} + a_1\xi^{n-1}+ \ldots +  a_{n-1}\xi + a_n \in \ann M.
\end{align*}
As $\xi^*$ is $G$-regular we have $\xi^n \in \A^{nr} \setminus \A^{nr+1}$.
 However  for each $i \geq 1$ we
 have
\[
a_i\xi^{n-i} \in \q^i\A^{(n-i)r} = \A^{i(r+1)}\A^{nr - ir} = \A^{nr + i} \subseteq
 \A^{nr +1}.
\]
Thus $u^* = (\xi^{n})^{*} = (\xi^*)^n$ (since $\xi^*$ is $G$-regular).
Set $u_i = u$ and $n_i = n$.
\end{proof}

The following Corollary is useful.
\begin{corollary}\label{redTOmaxDIM}
 Let $(A,\m)$ be \CM \ local ring, $\A$ an $\m$-primary ideal and let $M$ be an $A$-module. Set
$c = \dim A - \dim M$.  If $\GA$ is \CM \ then there exists $u_1,\ldots, u_c \in \A$ such that
\begin{enumerate}[\rm (1)]
 \item
$ u_1,\ldots, u_c \in \ann_A M $.
 \item
 $ u_{1}^*,\ldots, u_{c}^* \in \ann_{\GA} G_\A(M)$.
 \item
 $ u_1,\ldots, u_c $ is an $A$-regular sequence.
\item
$ u_{1}^*,\ldots, u_{c}^*$ is a $\GA$-regular sequence.
\item
 $u_1 t,\ldots, u_c t \in \ann_{\ra } \R(M).$
\item
$u_1 t,\ldots, u_c t$ is a $\ra$-regular sequence.
\end{enumerate}
\end{corollary}
\begin{proof}
 Since $\GA$ is \CM, we have that
\begin{align*}
  \grade  \ann_{\GA} G_\A(M)  &= \htt \ann_{\GA} G_\A(M) \\
                           &=  \dim \GA  - \dim G_{\A}(M),  \quad \text{since $\GA $ is *-local} \\
                           &= \dim A - \dim M .
\end{align*}
Therefore (1), (2), (3) and (4) follow from Lemma \ref{annihilator}.
The assertion (5) is clear.

(6). Since $t^{-1}$ is $\ra$-regular and $u_1^*,\ldots, u_c^*$ is $\GA = \ra/t^{-1}\ra$-regular sequence it follows that $t^{-1}, u_1t,\ldots, u_ct$ is  a $\ra$-regular sequence. Notice
$\ra$ is a *-local. It follows that $ u_1t,\ldots, u_ct$ is a $\ra$-regular sequence.
\end{proof}

We state our general result regarding CI-Approximations
\begin{theorem}\label{equi}
Let $(A,\m)$ be a complete
 local ring and let $\A$ be a proper  ideal in $A$ with $\dim A/\A + \spr(\A)
 = \dim A$. Then
$A$ has a CI-approximation \wrt \ $\A$. Furthermore if $\A = (x_1,\ldots,x_m)$ then we may
choose
a CI-approximation $[B,\n,\B,\varphi]$ of $[A,\m,\A]$ such that there exists $v_1,\ldots,v_m \in \B$ with $\varphi(v_i) = x_i$ for  $i = 1,\ldots, m$ and $\B = ( v_1,\ldots,v_m)$.
\end{theorem}
Theorem \ref{equi} is proved in section \ref{sectionGAPP}.
\begin{remark}\label{equiREM}
 The hypothesis of the theorem holds  when $\A$ is $\m$-primary. It is also satisfied when $\A$ is equimultiple and
$A$ is quasi-unmixed.   Our hypothesis ensures that $\GA$ has a homogeneous system of parameters,
  \cite[2.6]{HUO}.
\end{remark}

For the case when $\A = \m$ and $A$ is a quotient of a regular local ring $T$ we have:
\begin{theorem}\label{GorApprox-REG-QT}
 Let $(A,\m)$  be a local ring such that $A$ is a quotient of a regular local ring $(T,\tf)$. Then
$[A,\m,\m]$ has a CI-approximation $[B,\n,\n, \phi]$ with $\phi$-onto.
\end{theorem}
\begin{proof}
 Set $k = A/\m = T/\tf$. Let $\psi \colon T \rt A$ be the quotient map. We consider $A$ as  a
$T$-module. So $ \GA = G_{\tf}(A)$. Set $c = \dim T - \dim A = \dim \GT - \dim \GA$.
Its well-known  $\GT \cong k[X_1,\ldots,X_n]$; where $n = \dim T$. In particular $\GT$ is
\CM. Let $u_1,\ldots u_c$ be as in Corollary \ref{redTOmaxDIM}.

Set $(B,\n) = (T/(\ub),\tf/(\ub))$ and let $\phi \colon B \rt A$ be the map induced by
$\psi$. So $\phi$ is onto. Clearly
$[B,\n,\n, \phi]$ is a CI-approximation  of $[A,\m,\m]$.
\end{proof}

\noindent For  $\m$-primary ideals in a complete  equi-characteristic complete local ring we prove:
\begin{theorem}\label{field}
 Let $(A,\m)$ be a complete
equicharacteristic local ring of dimension $d$ and let $\mathfrak{a}$ be an
$\m$-primary ideal.  Then $[A,\m,\mathfrak{a}]$ has a CI-approximation $[B,\n,\mathfrak{b},\phi]$ with $\B = \n$ and $B/\n \cong A/\m$ and  $\mu(\n) =
\mu(\A)$. If $\mu(\A) > d$  and let $r \geq 1$ then we can choose CI-approximations $[B',\n,\mathfrak{b'},\phi']$ with $\B' = \n$ and $B/\n \cong A/\m$ and  $\mu(\n) =
\mu(\A)$ such that $\red(\B', B') > r$.
\end{theorem}
\begin{proof}
  $A$  contains its residue field $k = A/\m.$; see \cite[28.3]{Ma}.  Let
$\mathfrak{a} = (x_1, \ldots, x_n).$  Set $S = k [|X_1,\ldots X_n|]$
and let $\psi : S \rightarrow A$ be the natural map which sends
$X_i$ to $x_i$ for each $i.$
Since $A/(\underline{X}) A = A/\mathfrak{a}$ has finite
length (as an $S$-module) it follows that $A$ is a finitely
generated as an S-module, see \cite[8.4]{Ma}.  Set $\eta = (X_1, \ldots
X_n).$ Notice $\psi (\eta)A = \mathfrak{a}$ and $G_{\eta} (S) = k
[X^{*}_{1}, \ldots X^{*}_{n}]$ the polynomial ring in $n$-variables.

If $d= \dim S$ then set $[B,\n,\B,\phi] = [S,\eta,\eta,\psi].$
Otherwise set $c = \dim S-\dim A = \dim G_n (S) - \dim
G_{\mathfrak{a}} (A).$
Let $u_1,\ldots u_c$ be as in Corollary \ref{redTOmaxDIM}.
Set $B = S/(\ub)$, $\n = \eta/(\ub)$ and $\B = \n$. Clearly $G_{\n} (B) = G_{\eta}
(S)/(\ub^*)$ is CI.
 The map $\psi$ induces $\phi : B \rightarrow A.$  It can
be easily checked that $[B,\n,\n,\phi]$ is a CI-approximation of $[A,\m,\mathfrak{a}]$.
Finally by our construction its clear that $\mu(\n) = \mu(\A)$.

If $\mu(\A) > d$  then we simply let $B' = S/(u_1^s,u_2^s,\ldots, u_c^s)$ for $s$ large.
\end{proof}
\begin{remark}
If the reader is willing to work with equi-characteristic rings  and $\A$ is $\m$-primary then he/she need not read section 5.
\end{remark}

\section{A classical filtration of the dual}\label{section-classical}
Let $M$ be  an $A$-module.
The following  filtration of the dual of $M$, i.e, $M^{\dd} = \Hom_A(M, A)$ is classical cf.  \cite[p.\ 12]{Ser}.
Set
\[
M_{n}^{\dd} = \{ f \in M^{\dd} \mid f(M) \subseteq \A^n \} \cong \Hom_A(M,\A^n).
\]
It is easily verified that   $\eF_M = \{ M_{n}^{\dd} \}_{n \in \Z}$ is an $\A$-stable filtration on $M^{\dd}$.
We call $\eF_M$ to be the \emph{dual-filtration} of $M$ \wrt  \ $\A$.
Set $\R = \R(\A) = \bigoplus_{\nZ}\A^nt^n$  and
$\R(M) = \R(\A,M) = \bigoplus_{\nZ}\A^nMt^n$. Let $f \in  M_{n}^{\dd}$.  We show that $f$
induces a\textit{ homogeneous} $\R$-\textit{linear map} $\hat{f}$ \textit{of degree} $n$ \textit{from} $\R(M)$ to $\R$; (see \ref{mapAUG}).
Theorem  \ref{main}  shows that $\Psi_M$ is a $\R$-linear isomorphism.
\begin{align*}
\Psi_M \colon \R(\eF,M^{\dd}) &\xar \Hs_\R(\R(M),\R ) \\
f &\mapsto \f
\end{align*}
It is easily seen,  see \ref{initialobs},  that $\Psi_M$ induces a natural map
$$\Phi_M \colon G(\eF,M^{\dd}) \xar \Hs_{\GA}(G_{\A}(M),\GA). $$
We prove $\Phi_M$ is injective.
In Corollary \ref{corGor} we give a sufficient condition for $\Phi_M$ to be an isomorphism.

The following result is well-known.
\begin{proposition}
\label{stablefilt}
Let $(A,\m)$ be local, $\A$ an ideal in $A$ and let $M$ an  $A$-module.
Then $\A \Hom_{A}(M, \A^n) = \Hom_A(M,\A^{n+1})$ for all $n \gg 0$. \qed
\end{proposition}

\s \label{mapAUG} Let $f \in M_{n}^{\dd}$.

\noindent\textbf{Claim:} $f$ induces a homogeneous $\R$-linear map $\hat{f}$ of degree $n$ from $\R(M)$ to $\R$.

Since $f(M) \subseteq \A^n$, for each $j \geq 0 $ we have $f(\A^jM) \subseteq \A^{n+j}$. Furthermore for $j < 0 $ as $\A^jM = M$ and notice that $f(M) \subset \A^n \subset \A^{n+j}$.
Thus
\begin{equation*}
f(\A^jM) \subseteq \A^{n+j} \quad \text{for each $j \in \Z$}. \tag{$*$}
\end{equation*}
  This enables us to define
\begin{align*}
\hat{f} \colon \R(M) &\xar \R \\
\sum_{j \in \Z} m_jt^j &\mapsto \sum_{j \in \Z} f(m_j)t^{n+j}.
\end{align*}

Next we prove that $\hat{f}$ is $\R$-linear.
Clearly $\hat{f}$ is $A$-linear. Notice
\begin{align*}
 \f\left( (x_jt^j)\bullet(m_it^i) \right) &= \f(x_jm_it^{i+j}) = f(x_jm_i)t^{i+j+n} = x_jt^jf(m_i)t^{i+n}  \\
(x_jt^j)\bullet \f(m_it^i) &= x_jt^jf(m_i)t^{i+n}.
\end{align*}
Thus $\f$ is $\R$-linear.

\s \label{main-1} We define a map
\begin{align*}
\Psi_M \colon \R(\eF,M^{\dd}) &\xar \Hs_\R(\R(M),\R ) \\
f &\mapsto \f
\end{align*}

\begin{theorem}\label{main}
$\Psi_M \colon \R(\eF,M^{\dd}) \xar  ^*\Hom_{\R}(\R(M), \R)$ is a $\R$-linear isomorphism.
\end{theorem}
\begin{proof}
If $a \in \A^j$ and $f \in M_{i}^{\dd}$ then
$af \in M^{\dd}_{i+j}$. Check  that $(at^i)\bullet \hat{f} = \hat{af}$.
Thus $\Psi_M$ is $\R$-linear.
Clearly  if $\hat{f} = 0$ then $f = 0$. Thus $\Psi_M$ is injective.

We show $\Psi_M$ is surjective. Let $g \in \Hs_\R(\R(M),\R )_n$. We write $g$ as
\[
g\left(\sum_{j \in \Z}m_jt^j\right)  = \sum_{j \in \Z}g_j(m_j)t^{n+j}.
\]
For each $j \in \Z$ the map $g_j \colon \A^jM \rt \A^{n+j}$ is $A$-linear.
Define $f = i\circ g_0$ where $i \colon \A^n \rt A$ is the inclusion map. Clearly $f \in M_{n}^{\dd}$.

\textit{Claim:} $\hat{f} = g$.

Let $m_jt^j \in \A^jMt^j$.

Case 1. $j = 0$.

Set $m = m_0$. Notice
$$g(m t^0) = g_0(m)t^n = \hat{f}(m t^0).$$
The last equality above holds since $\hat{f}$ is $\R$-linear.

\noindent\textit{In the next two cases we use Case 1 and the fact that  $\hat{f}$ is $\R$-linear.}

Case 2. $j< 0$.

Notice $m_jt^j =  t^j \bullet m_jt^{0}$. So
$$g(m_jt^j) = t^j\bullet g(m_jt^0) =  t^j\bullet \hat{f}(m_jt^0) = \hat{f}(m_jt^j).$$

Case 3. $j >0$.

Set $m_j = \sum_{l = 1}^{s}u_{jl}n_l$ where $u_{jl} \in \A^j$ and $n_l \in M$.
Fix $l$.
Notice
$$ u_{jl}n_lt^j = u_{jl}t^j\bullet n_lt^0.$$
Therefore for $l = 1,\ldots, s$,
$$g(u_{jl}n_lt^j) = u_{jl}t^j\bullet g(n_l t^0) = u_{jl}t^j \bullet \hat{f}(n_lt^0) = \hat{f}(u_{jl}n_lt^j).$$
 Note the last equality above is since $\hat{f}$ is $\R$-linear. So we have
$$ g(m_jt^j) = g\left(\sum_{l = 1}^{s}u_{jl}t^jn_lt^0\right)
           = \sum_{l =1}^{s}g(u_{jl}t^jn_lt^0)
           = \sum_{l = 1}^{s}\hat{f}(u_{jl}t^jn_lt^0)
           = \hat{f}(m_jt^j). $$
  Again   the last equality above holds since $\hat{f}$ is $\R$-linear.
\end{proof}

\begin{observation}\label{initialobs}
$f \in M_{n}^{\dd}$ induces $\hat{f} \colon\R(M) \rt \R$ which is homogeneous of degree  $n$. So $\hat{f}$ induces a map $\tilde{f}\colon G_{\A}(M) \rt \GA$ which is also
homogeneous of degree $n$. Clearly $\tilde{f} = 0$ \ff \
$f \in M_{n+1}^{\dd}$. So we have
\begin{align*}
\Phi_M \colon G(\eF,M^{\dd}) &\xar \Hs_{\GA}(G_{\A}(M),\GA) \\
f + M_{n+1}^{\dd} &\mapsto \tilde{f}. \\
\text{Clearly} \quad \Phi_M &= \Psi_M\otimes \frac{\R}{t^{-1}\R}.
\end{align*}
\end{observation}

\begin{corollary}\label{basicCor}
Set $G = \GA$.
There is an exact sequence
\[
0 \xar G(\eF,M^{\dd}) \xrightarrow{\Phi_M} \Hs_{G}(G_{\A}(M),G) \xar \Es^{1}_{\R}(\R(M),\R )(+1)
\]
\end{corollary}
\begin{proof}
The map $ 0\rt \R(+1) \xrightarrow{t^{-1}} \R \rt G \rt 0$ induces the exact sequence
\begin{align*}
0&\xar \Hs_{\R}(\R(M),\R)(+1) \xrightarrow{t^{-1}} \Hs_{\R}(\R(M),\R) \xar \Hs_{\R}(\R(M),G)\\
\ &\xar \Es_{\R}^{1}(\R(M),\R)(+1)
\end{align*}
Note that $\Hs_{\R}(\R(M),G) \cong \Hs_{G}(G_{\A}(M),G)$. The result follows by using
Theorem \ref{main} and Observation \ref{initialobs}.
\end{proof}

An important consequence of  Corollary \ref{basicCor} is the following:
\begin{corollary}\label{corGor}
Let $(A,\m)$ be a Gorenstein local ring and let $\A$ be an  ideal such that $\GA$ is  Gorenstein. Let $M$ be a maximal \CM \ $A$-module with $G_{\A}(M)$-\CM. Then
$$ G(\eF,M^{\dd}) \cong \Hs_{\GA}(G_{\A}(M),\GA).$$
\end{corollary}
\begin{proof}
Notice $\R$ is also a Gorenstein ring and $\R(M)$ is maximal \CM. It follows that
$\Es_{\R}^{1}(\R(M),\R) = 0$. The result follows from Corollary \ref{basicCor}.
\end{proof}

We now discuss dual filtration's and the Veronese functor. We show
\begin{proposition}
  \label{dual-Ver-prop}
  Let $(A,\m)$ be  local, $\A$ an ideal in $A$ and let $M$ be an $A$-module. Let $\eF$ be the dual filtration of $M$ \wrt  \ $\A$. Then for all $l\geq 1$;  $\eF^{<l>}$ is the dual filtration of $M$ \wrt \ $\A^l$.
\end{proposition}
\begin{proof}
Fix $l \geq 1$. Let $\eG$ be the dual filtration of $M$ \wrt  \ $\A^l$.
  We note that $\eF_{nl} = \Hom_A(M, \A^{nl}) = \eG_{n}$. The result follows.
\end{proof}

\section{CI approximation: The general case} \label{sectionGAPP}
In this section we prove  our general result regarding  complete intersection (CI)
approximation; see  \ref{equiTT}. In  \ref{field} we showed that every $\m$-primary ideal $\A$ in an equicharacteristic
local ring $A$ has a
CI-approximation. Even if we are interested only in the case of $\m$-primary ideals; we have to take some care for
dealing with the case of local rings with mixed characteristics. \emph{The essential point is to show existence of
homogeneous regular sequences in certain graded ideals.}

This section is divided into two subsections. In the first subsection $R = \bigoplus_{n\geq 0}R_n$ be  a standard graded algebra
over a local ring $(R_0,\m_0)$. When $R$ is \CM \ and $M =\bigoplus_{n\geq 0}M_n$ is a \fg graded $R$-module generated by elements in $M_0$,
we give conditions to ensure a  homogeneous regular sequence $\xi_1,\ldots,\xi_c \in R_+ \cap\ann_R M$
where $c= \dim R - \dim M$ (see Theorem \ref{hann}).
In the second subsection we prove Theorem \ref{equiTT}.

\textbf{Homogeneous regular sequence }

Let  $R = \bigoplus_{n\geq 0}R_n$ be  a standard  algebra
over a local ring $(R_0,\m_0)$ and let $M =\bigoplus_{n\geq 0}M_n$ be a \fg graded $R$-module.
 In general a graded module $M$
need not have homogeneous  regular sequence $\xi_1,\ldots,\xi_c \in R_+ \cap\ann_R M$ where $c = \grade M$ (see \ref{eqEx}).
 When $R$ is \CM \ and $M =\bigoplus_{n\geq 0}M_n$ is a \fg \ graded $R$-module generated by elements in $M_0$,
we give conditions to ensure a homogeneous regular sequence $\xi_1,\ldots,\xi_c \in R_+ \cap\ann_R M$
where $c= \dim R - \dim M$.

We adapt an example from \cite[p.\ 34]{BH} to show that a homogeneous  regular sequence of length
$\grade M$
in $R_+ \cap\ann_R M$ need not exist.

\begin{example}\label{eqEx}
 Let $A = k[[X]]$ and $T = A[Y]$. Set $R = T/(XY)$. Notice that $R_0 = A$. Set $\M = (X,Y)R$ the unique graded maximal ideal of $R$ and let $M = R/\M = k$. Clearly $\grade M = \grade(\M , R) = 1$. It can be easily checked that
every homogeneous element in $R_+$ is a zero-divisor.
\end{example}

The following result regarding existence of homogeneous regular sequence in $(\ann M) \cap R_+$ (under certain conditions) is crucial in our proof of Theorem \ref{equiTT}.
\begin{theorem}
\label{hann}
Let $(R_0,\m_0)$ be a  \ local ring and let $R = \bigoplus_{n\geq 0}R_n$ be a
standard graded $R_0$-algebra and let $M = \bigoplus_{n\geq 0}M_n$ be a finitely generated
$R$-module. Assume $R$ has a homogeneous s.o.p. Set $c = \dim R - \dim M$.
Assume
\begin{enumerate}[ \rm (1)]
\item
$R_0$ is \CM \ and $R$ is \CM.
\item
$M$ is generated as an $R$-module by some elements in $M_0$.
\item
There exists
$y_1,\ldots,y_s \in R_0$  a system of parameters for both $M_0$ and $R_0$ and a  part of a homogeneous system of parameters for both $M$ and $R$.
\end{enumerate}
Then there exists a homogeneous regular sequence
  $\xi_1,\ldots,\xi_c \in R_+$ such that $\xi_i \in \ann_R M$ for each $i$.
\end{theorem}
\begin{proof}
We prove this by induction on $s = \dim R_0$.
When $s = 0$,  $R_0$ is Artinian. So $\grade(\ann M,R) = \grade(\ann M \cap R_+, R)$. As $R$ is CM \ we have $\grade(\ann M, R) = c$.
 The result  follows from \cite[1.5.11]{BH}.

We  prove the assertion for $s = r +1$
assuming it to be valid for $s=r$. Set $S = R/y_1R = \bigoplus_{n\geq 0} S_n$ and $N = M/y_1M$. Note that $N$ is also generated in degree $0$.  Let $\ov{y_i} = y_i \mod (y_1)$ with $2 \leq i \leq s$. Then $\ov{y_2} \ldots \ov{y_s}$
is a system of parameters for both $S_0$ and $N_0 = M_0/y_1M$. Notice $c = \dim R - \dim M = \dim S - \dim N$. Also note that $S$ has h.s.o.p. By induction hypothesis
there exists a $S$-regular sequence
  $\ov{\xi_1},\ldots,\ov{\xi_c} \in S_+$ such that $\ov{\xi_i} \in \ann_S N$ for each $i$.

$R$ is CM. So $y_1$ is  $R$-regular.
Therefore $y_1,\xi_1,\xi_2\ldots,\xi_c$ is an $R$-regular sequence.
Set $\q = y_1R_+$. Let $M$ be generated as an $R$-module by $u_1,\ldots,u_m \in M_0$.
Fix $i \in \{1,\dots,c \}$ and set $\xi = \xi_i$. By construction $\xi u_j \in y_1 M$ for each
 $j$. However $\xi \in R_+$ and $y_1 \in R_0$ . So $\xi u_j \in \q M$ for each
 $j$. By the determinant trick (observing that each $u_j$ has  degree zero)
 there exists a \textit{ homogeneous} polynomial
\[
\Delta(\xi) = \xi^{n} + a_{1}\xi^{n-1} + \ldots +a_{n-1}\xi + a_n \in \ann_R M \quad \text{and} \ a_i \in \q^i \ \text{for} \ i =1,\dots,n.
\]
Clearly $\Delta(\xi) \in R_+$.
Notice that $\Delta(\xi) = \xi^{n} \mod (y_1)$. So
$y_1,\Delta(\xi_1),\ldots,\Delta(\xi_c)$ is a regular sequence, \cite[16.1]{Ma}.
Therefore $ \Delta(\xi_1),\Delta(\xi_2)\ldots,\Delta(\xi_c)$
is an $R$-regular sequence. This proves the assertion when $s =r+1$.
\end{proof}

\textbf{CI approximation}

In this subsection we prove our general result regarding  CI-approximation.

\begin{theorem}\label{equiTT}
Let $(A,\m)$ be a complete
 local ring and let $\A$ be a proper  ideal in $A$ with $\dim A/\A + \spr(\A)
 = \dim A$. Then
$A$ has a CI-approximation \wrt \ $\A$. Furthermore if $\A = (x_1,\ldots,x_m)$ then we may
choose
a CI-approximation $[B,\n,\B,\varphi]$ of $[A,\m,\A]$ such that  $\B = (v_1,\ldots,v_m)$ with $\varphi(v_i) = x_i$ for  $i = 1,\ldots, m$.
\end{theorem}
\begin{proof}
Notice $\GA$ has a homogeneous s.o.p (see proof of Proposition 2.6 in \cite{HUO}).
 Let $ y_1,\ldots,y_s \in A$ such that $ \ov{y_1},\ldots,\ov{y_s}$ (their images in $A/\A$ ) is an
\emph{s.o.p} of $A/\A$.

\textbf{Case 1.} \emph{A contains a field:}
 As $A$ is a complete,
$A$ contains  $k = A/\m$.
  Set $S = k[[ Y_1,\ldots, Y_s, X_1,\ldots,X_n ]]$. Define
$\phi \colon S \rt A$ which maps $Y_i$ to $y_i$ and $X_j$ to $x_j$ for $i= 1,\ldots, s$
and $j = 1,\ldots, n$.
 Set $\q = (Y_1,\ldots, Y_s)$.
Notice
\[
\frac{A}{(\q, X_1,\ldots,X_n)A} = \frac{A}{(\A, y_1,\ldots,y_s)} \quad \text{has finite length.}
\]
As $S$ is complete we get that $A$ is a finitely generated $S$-module, cf. \cite[8.4]{Ma}. Since
$\psi((\bX))A = \A$,  we get  $\GA = G_{(\bX)}(A)$ is a finitely generated
$G_{(\bX)}(S)$-module.

Notice $G_{(\bX)}(S) \cong
S/(\bX)[X_{1}^{*},\ldots, X_{n}^{*}]$ and $S/(\bX) \cong k[| Y_1,\ldots, Y_s|]$.
Furthermore
 $G_{(\bX)}(A) = \GA$ is generated in degree zero as a $G_{(\bX)}(S)$-module.
Set $c =  \dim S - \dim A  = \dim G_{(\bX)}(S)  - \dim \GA $.  Notice  that $ Y_1,\ldots, Y_s$ is a s.o.p of  $G_{(\bX)}(S)_0$ and $\GA_0$.
We apply Theorem \ref{hann} to get $\xi_{1},\ldots,\xi_{c} \in (\bX)  $ such that
 $\xi_{1}^{*}, \ldots, \xi_{c}^{*} \in G_{(\bX)}(S)_+  $ is a \emph{regular sequence}   in  $G_{(\bX)}(S)$ and
$\xi_{i}^{*} \in \ann_{G_{(\bX)}(S)} \GA$ for all $i$. We now apply  Lemma
\ref{annihilator} to get $u_1,\ldots, u_c  \in (\bX)$ an $S$-regular sequence in
$\ann_S (A)$ such that $u_{1}^{*},\ldots, u_{c}^{*}$ is a $G_{(\bX)}(S)$-regular
 sequence in $\ann_{G_{(\bX)}(S)} \GA $.  Set $\mathbf{u} = (u_1,\ldots, u_c) $,   $B = S/\mathbf{u} $, $\B = (\bX +\mathbf{u})/\mathbf{u}$.  Let $\n$ be the maximal ideal in $B$ and let $\ov{\phi} \colon B \rt A$
be the map induced by $\phi \colon  S \rt A$. Then it is clear that
$(B,\n,\B, \ov{\phi})$ is a CI approximation of $A$ \wrt \ $\A$. Set $v_i = \ov{X_i}$. Then $\B = (v_1,\ldots,v_m)$ with $\varphi(v_i) = x_i$ for  $i = 1,\ldots, m$.

\textbf{Case 2.} \textit{$A$ does not contain a field:} There exists a DVR, $(D,\pi)$, and a local
ring homomorphism $\eta \colon D \rt A$ such that $\eta$ induces an isomorphism
 $D/(\pi) \rt A/\m$.  Set $S = D[| Y_1,\ldots, Y_s, X_1,\ldots,X_n|]$. Define
$\phi \colon S \rt A$, with  $\phi(d) =  \eta(d)$ for each $d \in D$ and maps  $Y_i$ to $y_i$ and $X_j$ to $x_j$ for $i= 1,\ldots, s$
and $j = 1,\ldots, n$.
 Set $\q = (Y_1,\ldots, Y_s)$.
Notice
\[
\frac{A}{(\q, X_1,\ldots,X_n)A} = \frac{A}{(\A, y_1,\ldots,y_s)} \quad \text{has finite length.}
\]
As $S$ is complete we get that $A$ is a finitely generated $S$-module.

Set $T = S/(\bX) = G_{\bX}(S)_0$. Notice $\dim T$ is one more than that of $A/\A = \GA_0$.
To deal with this situation we  use an argument from \cite{PuEuler}, which we repeat for the convenience of the reader. Let $\mathfrak{t}$ be the maximal ideal of $T$. Set $\ov{Y_i} = $ image of $Y_i$ in $T$. Set $V= A/\A$.
Note that
\[
\mathfrak{t} = \sqrt{\ann_T(V/\ov{\bY} V)} = \sqrt{\ann_T(V) + (\ov{\bY})}.
\]
So there exists $\ov{f} \in \ann_T(V) \setminus (\ov{\bY})$ such
that $\ov{f},\ov{Y_1},\ldots,\ov{Y_s}$ is an s.o.p. of $T$. Since $T$ is CM; $f,\ov{Y_1},\ldots,\ov{Y_s}$ is a $T$-regular
sequence.
So $X_1,\ldots,X_n,f,Y_1,\ldots,Y_s$ is a $S$-regular
sequence. Since $\ov{f} \in \ann_T(A/\A)$  we get $f A \subseteq \A A = (\bX)A$. Using the
determinant trick
 there exists
\[
\Delta = f^m + \alpha_{1}f^{m-1} + \ldots + \alpha_{m-1}f + \alpha_{m} \in \ann_S A \quad \& \alpha_i \in (\bX)^i; \ 1\leq i \leq m.
\]
 Notice that $\Delta = f^m \ \text{mod}( \bX)$. Thus
 $X_1,\ldots,X_n, \Delta,Y_1,\ldots,Y_s$ is a $S$-regular
sequence. So $\Delta$ is $S$-regular. Set $U = S/( \Delta)$ and $\C = ((\bX) + ( \Delta))/( \Delta)$.
 Notice that $A$ is a finitely generated $U$-module and $G_{\C}(A) = \GA$. Furthermore
$G_{\C}(U) \cong T/(f^m)[X_{1}^{*},\ldots, X_{n}^{*}]$ is  CI.

 Note that as a
 $G_{\C}(U)$ module $\GA $ is generated in degree zero. Furthermore $\dim G_{\C}(U)_0 = \dim \GA_0$. Set $c =  \dim U - \dim A  = \dim G_{\C}(U)  - \dim \GA $.
The subsequent argument is similar to  that of Case 1.
\end{proof}
\begin{remark}\label{high-red} It follows from the proof of Theorem \ref{equiTT} that if $\A$ is $\m$-primary
with $\mu(\A) \geq d + 1$ then
there exists CI-approximations with arbitrary high reduction numbers. To see this note that
there exist a complete intersection $R$ (with $G(R)$ a complete intersection) and a   map $R \rt A$ and elements $f_1,\cdots, f_m \in  \ann_R A$ such that $f_1^*,\cdots, f_m^* \in \ann_{G(R)} \GA$ and $f_1^*,\cdots, f_m^*$ is $G(R)$-regular. Furthermore $B = R/(f_1,\cdots, f_m)$. Also note that $\deg f_i^* > 0$ for all $i$. We may simply take $B'$ to be
$R/(f_1^l,\cdots, f_m^l)$ for large $l$.

\end{remark}

\section{On $L^I(M)$}\label{Lprop}
In \cite{Pu5} we introduced a new technique to investigate problems relating to associated graded modules.
In this section we collect all the relevant results which we proved in \cite{Pu5}. Throughout this section
$(A,\m)$ is a Noetherian local ring with infinite residue field, $M$ is a \emph{\CM }\ module of dimension
$r \geq 1$ and $I$ is an $\m$-primary ideal.

\s \label{mod-struc} Set $\Sc = A[It]$;  the Rees Algebra of $I$. Set
$L^{I}(M) = \bigoplus_{n\geq 0}M/I^{n+1}M$. We note that $L^I(M)(-1) =  M[t]/\R(I, M)$.  So  $L^{I}(M)$ is a $\Sc$-module. Note that $L^I(M)$ is \emph{not} a finitely generated $\Sc$-module.

 \s Set $\M = \m\oplus \Sc_+$. Let $H^{i}(-) = H^{i}_{\M}$ denote the $i^{th}$-local cohomology functor \wrt \ $\M$. Recall a graded $\Sc$-module $L$ is said to be
*-Artinian if
every descending chain of graded submodules of $L$ terminates. For example if $E$ is a finitely generated $\Sc$-module then $H^{i}(E)$ is *-Artinian for all
$i \geq 0$.

\begin{definition}(\cite[sect. 6]{HeL})
   Consider the following chain of submodules of
$M$:
\[
IM \sub (I^2M\colon_M I) \sub (I^3M\colon_M I^2) \sub \ldots \sub(I^{n+1}M \colon_M I^n)\sub \ldots
\]
As $M$ is Noetherian this chain stabilizes.
The stable value is denoted as $\widetilde{IM}$ and is called
 the \emph{Ratliff-Rush submodule of $M$ \wrt \ $I$}.
 The filtration
$\{\wt{I^nM}\}_{n \geq 1}$ is called the \textit{Ratliff-Rush filtration} of $M$
\wrt \ $I$.
\end{definition}

\s \label{RRdefn} It can be shown that if $\depth M >0$ then $\wt{I^nM} = I^nM$ for all $n \gg 0$.
Furthermore if $M = A$ then $\wt{I} \subseteq \ov{I}$,  where $\ov{I}$ is the integral closure of $I$. We say an ideal $I$ is Ratliff-Rush if $\wt{I} = I$. Note if $I$ is $\m$-primary and depth of $A$ is positive then $\depth G_I(A) \geq 1$ if and only if  $I^n$ is Ratliff-Rush for all $n \geq 1$.

\s \label{zero-lc} In \cite[4.7]{Pu5} we proved that
\[
H^{0}(L^I(M)) = \bigoplus_{n\geq 0} \frac{\wt{I^{n+1}M}}{I^{n+1}M}.
\]
\s \label{Artin}
For $L^I(M)$ we proved that for $0 \leq i \leq  r - 1$
\begin{enumerate}[\rm (a)]
\item
$H^{i}(L^I(M))$ are  *-Artinian; see \cite[4.4]{Pu5}.
\item
$H^{i}(L^I(M))_n = 0$ for all $n \gg 0$; see \cite[1.10 ]{Pu5}.
\item
 $H^{i}(L^I(M))_n$  has finite length
for all $n \in \mathbb{Z}$; see \cite[6.4]{Pu5}.
\item
$\ell(H^{i}(L^I(M))_n)$  coincides with a polynomial for all $n \ll 0$; see \cite[6.4]{Pu5}.
\end{enumerate}

\s \label{I-FES} The natural maps $0\rt I^nM/I^{n+1}M \rt M/I^{n+1}M \rt M/I^nM \rt 0 $ induce an exact
sequence of $\Sc$-modules
\begin{equation}
\label{dag}
0 \xar G_{I}(M) \xar L^I(M) \xrightarrow{\Pi} L^I(M)(-1) \xar 0.
\end{equation}
We call (\ref{dag}) \emph{the first fundamental exact sequence}.  We use (\ref{dag}) also to relate the local cohomology of $G_I(M)$ and $L^I(M)$.

\s \label{II-FES} Let $x$ be  $M$-superficial \wrt \ $I$ and set  $N = M/xM$ and $u =xt \in \Sc_1$. Notice $L^I(M)/u L^I(M) = L^I(N)$.
For each $n \geq 1$ we have the following exact sequence of $A$-modules:
\begin{align*}
0 \xar \frac{I^{n+1}M\colon x}{I^nM} \xar \frac{M}{I^nM} &\xrightarrow{\psi_n} \frac{M}{I^{n+1}M} \xar \frac{N}{I^{n+1}N} \xar 0, \\
\text{where} \quad \psi_n(m + I^nM) &= xm + I^{n+1}M.
\end{align*}
This sequence induces the following  exact sequence of $\Sc$-modules:
\begin{equation}
\label{dagg}
0 \xar \Bcal^{I}(x,M) \xar L^{I}(M)(-1)\xrightarrow{\Psi_u} L^{I}(M) \xrightarrow{\rho^x}  L^{I}(N)\xar 0,
\end{equation}
where $\Psi_u$ is left multiplication by $u$ and
\[
\Bcal^{I}(x,M) = \bigoplus_{n \geq 0}\frac{(I^{n+1}M\colon_M x)}{I^nM}.
\]
We call (\ref{dagg}) the \emph{second fundamental exact sequence. }

\s \label{long-mod} Notice  $\ell\left(\Bcal^{I}(x,M) \right) < \infty$. A standard trick yields the following long exact sequence connecting
the local cohomology of $L^I(M)$ and
$L^I(N)$:
\begin{equation}
\label{longH}
\begin{split}
0 \xar \Bcal^{I}(x,M) &\xar H^{0}(L^{I}(M))(-1) \xar H^{0}(L^{I}(M)) \xar H^{0}(L^{I}(N)) \\
                  &\xar H^{1}(L^{I}(M))(-1) \xar H^{1}(L^{I}(M)) \xar H^{1}(L^{I}(N)) \\
                 & \cdots \cdots \\
               \end{split}
\end{equation}

\s \label{Artin-vanish} We will use the following well-known result regarding *-Artinian modules quite often:

Let $V$ be a *-Artinian $\Sc$-module.
\begin{enumerate}[\rm (a)]
\item
$V_n = 0$ for all $n \gg 0$
\item
If $\psi \colon V(-1) \rt V$ is a monomorphism then $V = 0$.
\item
If $\phi \colon V \rt V(-1)$ is a monomorphism then $V = 0$.
\end{enumerate}

\s \label{normal}
Recall an ideal $I$ is said to be asymptotically normal if $I^n$ is integrally closed for all $n \gg 0$.
If $I$ is a asymptotically normal  $\m$-primary ideal  and $\dim A \geq 2$ then by a result of Huckaba and Huneke \cite[3.1]{HH}, $\depth G_{I^l}(A) \geq 2$
for all $l \gg 0$(also see \cite[7.3]{Pu5}). It  also follows from \cite[9.2]{Pu5} that in this case $H^1(L)_n = 0$ for $n < 0$. In particular $\ell(H^1(L)) < \infty$ (here $L = L^I(M)$).

\s Let $\eF = \{ M_n \}_{n \in \Z}$ be an $I$-stable filtration with $M_n = M$ for $n \leq 0$.  Set $L^\eF(M) = \bigoplus_{n \geq 0} M/M_{n+1}$. Then as before $L^\eF(M)$ is a $\Sc$-module.
All the results proved for $L^I(M)$ in the $I$-adic case can be proved similarly for $L^\eF(M)$ with the same proofs.

\s \label{asympCM} We will also need the following fact for general filtration's (and was proved for $I$-adic filtration's in \cite[9.2]{Pu5}). Let $\eF = \{ M_n \}_{n \in \Z}$ be an $I$-stable filtration with $M_n = M$ for $n \leq 0$. If the associated graded module of the Veronese $G(\eF^{< c >}, M)$ has depth $\geq 2$ for some $c \geq 1$ then
$H^1(G(\eF, M))_n = 0$ for $n < 0$.

To see this set $L(\eF, M)= \bigoplus_{n \geq 0} M/M_{n+1}$.  For all $l \geq 1$ we have
\[
\left(L(\eF,M)(-1)\right)^{<l>} = \bigoplus_{n \geq 0} M/M_{nl} =  L(\eF^{<l>},M)(-1)
\]
As $\depth G(\eF^{< c >}, M) \geq 2$ it follows from an analogue of \ref{Artin}; \ref{dag} for filtration's and \ref{Artin-vanish} that
$H^i(L(\eF^{<c>},M)) = 0$ for $i = 0,1$. As the Veronese functor commutes with local cohomology we get that
$H^i(L(\eF, M))_{nc-1} = 0$ for all $n \in \Z$. In particular $H^1(L(\eF, M))_{-1} = 0$. Let $x$ be $M$-superficial \wrt \ $\eF$.
Set $N = M/xM$. Then by an analogue of \ref{longH} to filtration's we have  an exact sequence for all $n \in \Z$
\[
H^0(L(\ov{\eF}, N))_n \rt H^1(L(\eF, M))_{n-1} \rt H^1(L(\eF, M))_{n}
\]
As $H^0(L(\ov{\eF}, N))_n = 0$ for $n < 0$ and $H^1(L(\eF, M))_{-1} = 0$;  an easy induction yield's $H^1(L(\eF, M))_{n} = 0$ for $n < 0$.
By the analogue of  exact sequence \ref{dag} and as $H^0(L(\eF, M)_n = 0$ for $n < 0$ we get that $H^1(G(\eF, M))_n = 0$ for $n < 0$.
\section{ Proof of Itoh's-conjecture for normal ideals}
In this section we give a proof of Itoh's conjecture for normal ideals.
\begin{theorem}\label{itoh}
 Let $(A,\m)$ be a \CM \  local ring of dimension $d \geq 1$ and let $\A$ be a normal $\m$-primary ideal.
 If $e_3^\A(A) = 0$ then $G_\A(A)$ is \CM.
\end{theorem}

\begin{remark}\label{r1-itoh}
The following assertion is well-known.
 If $\A^n$ is integrally closed for all $n$ then $\A^n$ is Ratliff-Rush for all $n \geq 1$.
 It follows that $\depth G_\A(A) \geq 1$, see \ref{RRdefn}.
\end{remark}

We first show that
\begin{lemma}\label{small}
 Theorem \ref{itoh} holds  if $\dim A = 1, 2$.
 \end{lemma}
 We note that $e_3^\A(A)$ is defined as in  \ref{hilb-gt-d}.
 This Lemma is certainly well-known to the experts. However we provide a proof
 for the convenience of the reader.
\begin{proof}[Proof of Lemma \ref{small}]
 If $d = 1$ then by \ref{r1-itoh} the result holds.

 Now assume $d = 2$. Note that we may assume that $A$ has an infinite residue field. Let $\C$ be a minimal reduction of $\A$.
 Set $\sigma_i = \ell(\A^{i+1}/\C \A^i)$. Then  as $\depth G_\A(A)$ is positive it follows from a result of Huneke \cite[2.4]{Hun}
 that the $h$-polynomial of $A$ \wrt \ $\A$ is
 \[
  h(t) =  \ell(A/\A) + (\sigma_0 - \sigma_1)t + (\sigma_1 - \sigma_2)t^2 + \cdots + (\sigma_{s-2} - \sigma_{s-1})t^{s-1} + \sigma_{s-1} t^s.
 \]
It follows that
\begin{align*}
 e_1^\A(A) &= \sum_{i \geq 0} \sigma_i, \\
 e_2^\A(A) &= \sum_{i \geq 1} i \sigma_i, \quad  \text{and} \\
 e_3^\A(A) &= \sum_{i \geq 2} \binom{i}{2} \sigma_i.
 \end{align*}
As $e_3^\A(A) = 0$ we have $\sigma_2 = 0$. So we have $\A^3 = \C \A^2$. As $\A$ is integrally closed  we also have
$\A^2 \cap \C = \C \A$, see \cite[Theorem 4.7 ]{Hun} (for rings containing a field) and \cite[Theorem 1]{It}. The result follows from Valabrega-Valla Theorem
\cite[2.3]{VV}.
\end{proof}
\s Recall that an ideal $\A$ is said to be \textit{asymptotically normal} if $\A^n$ is integrally closed for all $n \gg 0$.
The crucial result to prove Itoh's conjecture is the following:
\begin{lemma}\label{crucial-itoh}
Let $(A,\m)$ be a \CM \    local ring of dimension $3$ and let $\A$ be an asymptotically  normal $\m$-primary ideal with $e_3^\A(A) = 0$. Set $L^\A(A) = \bigoplus_{n \geq 0}A/\A^{n+1}$ considered as a  module over the Rees algebra $\Sc = A[\A t]$. Let $\M = \m\oplus \Sc_+$ be the maximal homogeneous ideal of $\Sc$. Then the local cohomology modules $H^i_\M(L^\A(A))$ vanish for $i = 1, 2$
\end{lemma}
We will also need to extend Lemma \ref{crucial-itoh} to dimensions $d \geq 4$.
\s \emph{Remark and a Convention:} Note that all the relevant graded modules considered upto Lemma \ref{rachel-itoh}  below are modules over the Rees algebra $\Sc = A[\A t]$.  Also all local cohomology will
taken over $\M = \m\oplus \Sc_+$  the maximal homogeneous ideal of $\Sc$. Note $G_\A(A)$ is a quotient of $\Sc$.
We also note that if $x$ is $A$-superficial \wrt \ $A$ then the Rees algebra $\Sc^\prime  = A/(x)[\A/(x)t]$ is a quotient of $\Sc$. As we are only interested in vanishing of certain local-cohomology modules, by the independence theorem of local cohomology
it does not matter if we take local cohomology of an $\Sc^\prime$-module \wrt \ $\M^\prime$ or over $\M$ (and considering the module in question as an $\Sc$-module). So throughout we will only write $H^i(-)$ to mean $H^i_\M(-)$.

\begin{lemma}\label{crucial-itoh-2}
Let $(A,\m)$ be a \CM \ local ring of dimension $d \geq 3$ and let $\A$ be an asymptotically  normal $\m$-primary ideal with $e_3^\A(A) = 0$.  Then the local cohomology modules $H^i(L^\A(A))$ vanish for $1 \leq i \leq d -1$.
\end{lemma}
\begin{proof}
We prove the result by induction on $d \geq 3$.
For $d = 3$ the result follows from Lemma \ref{crucial-itoh}. Now assume $d \geq 4$ and the result has been proved for $d-1$.
After passing through a general extension (see \ref{AtoA'}(d)) we may choose $x$, an $A$-superficial element
\wrt  \ $\A$ such that in the ring $B = A/(x)$ the ideal $\B = \A/(x)$ is asymptotically normal. Furthermore note that $e_3^\B(B) = 0$.
Set $L = L^\A(A)$ and $\ov{L} = L^\B(B)$. By induction hypothesis we have $H^i(\ov{L}) = 0$ for $1 \leq i \leq d-2$. By \ref{longH} we get a surjective map
$H^1(L)(-1) \rt H^1(L)$ and for $2 \leq i \leq d-1$ injections
$H^i(L)(-1) \rt H^i(L)$. By \ref{Artin-vanish} it follows that $H^i(L) = 0$ for $i = 2,\cdots d - 1$. By \ref{normal} we get that $H^1(L)$ has finite length. So the surjection
$H^1(L)(-1) \rt H^1(L)$ induces an isomorphism $H^1(L)(-1) \cong H^1(L)$ and this forces $H^1(L) = 0$.
\end{proof}

We now give a proof of Theorem \ref{itoh} assuming Lemma \ref{crucial-itoh}.
\begin{proof}[Proof of Theorem \ref{itoh}]
If $d = \dim A = 1,2$, the result follows from Lemma \ref{small}. If $d \geq 3$
then by Lemma \ref{crucial-itoh-2} we get $H^i(L^\A(A)) = 0$  for $1 \leq i \leq d -1$. Also as $\A$ is normal,  in-particular $\A^n$ is Ratliff-Rush for all
$n \geq 1$. So by \ref{zero-lc} we get $H^0(L^\A(A)) = 0$. By taking cohomology of the first fundamental sequence \ref{dag} we get that $H^i(G_\A(A)) = 0$
for  $0 \leq i \leq d -1$. Thus $G_\A(A)$ is \CM.
\end{proof}

\s \label{red-itoh} Thus to prove Itoh's conjecture all we have to do is to prove Lemma \ref{crucial-itoh}. This requires several preparatory results.
\emph{For the rest of this section we will assume $\dim A = 3$ and $\A$ is an asymptotically normal ideal with $e_3^\A(A) = 0$}. We will also assume that the residue field of $A$ is infinite.
We first show
\begin{lemma}\label{asymp-itoh}
(with hypotheses as in \ref{red-itoh}.) Then $G_{\A^n}(A)$ is \CM \ for $n \gg 0$
and $\red(\A^n) \leq 2$ for $n \gg 0$.
\end{lemma}
\begin{remark}
Lemma \ref{asymp-itoh} is certainly known to the experts. We provide a proof for the convenience of the reader.
\end{remark}
\begin{proof}[Proof of Lemma \ref{asymp-itoh}]
Let $\C = (x_1, x_2, x_3)$ be a minimal reduction of $\A$.
Notice that for all $n \geq 1$ we have $e_3^{\A^n}(A) = e_3^\A(A) = 0$. Furthermore $\C^{[n]} = (x_1^n, x_2^n, x_3^n)$ is  a minimal reduction of $\A^n$.
By a result of Huckaba and Huneke \cite[3.1]{HH}  we have $\depth G_{\A^n}(A) \geq 2$
for all $n \gg 0$, say $n \geq n_0$. As $\A$ is asymptotically normal we may also assume $\A^n$ is integrally closed for all $n \geq n_0$

Fix $n \geq n_0$.  Set $\sigma_i = \ell( (\A^n)^{i+1}/ \C^{[n]} (\A^n)^{i})$. Then by a result of Huckaba \cite[2.11]{Huck}, it follows that
that the $h$-polynomial of $A$ \wrt \ $\A^n$ is
 \[
  h(t) =  \ell(A/\A^n) + (\sigma_0 - \sigma_1)t + (\sigma_1 - \sigma_2)t^2 + \cdots + (\sigma_{s-2} - \sigma_{s-1})t^{s-1} + \sigma_{s-1} t^s.
 \]
It follows that
\begin{align*}
 e_1^{\A^n}(A) &= \sum_{i \geq 0} \sigma_i, \\
 e_2^{\A^n}(A) &= \sum_{i \geq 1} i \sigma_i, \text{and} \\
 e_3^{\A^n}(A) &= \sum_{i \geq 2} \binom{i}{2} \sigma_i.
 \end{align*}
As $e_3^{\A^n}(A) = 0$ we have $\sigma_2 = 0$. So we have $(\A^n)^3 = \C^{[n]} (\A^n)^2$. Thus $\red(\A^n) \leq 2$ for $n \geq n_0$. As $\A^n$ is integrally closed  we also have
$(\A^n)^2 \cap \C^{[n]} = \C^{[n]} \A^n$.  So by Valabrega-Valla Theorem \cite[2.3]{VV} we have that $G_{\A^n}(A)$ is \CM.
\end{proof}
Next we show
\begin{lemma}\label{rachel-itoh}
(with hypotheses as in \ref{red-itoh}.) We have
\begin{enumerate}[ \rm(1)]
\item
$a(G_\A(A)) < 0$.
\item
$H^2(L^\A(A)) = 0$.
\item
$G_\A(A)$ is generalized \CM.
\item
$\sum_{i = 0}^{2}(-1)^i\ell(H^i(G_\A(A)))  = 0.$
\item
For $i = 0, 1, 2$ we have $H^i(G_\A(A))_n = 0$ for $n < 0$. Furthermore $H^2(G_\A(A))_0 = 0$.
\end{enumerate}
\end{lemma}
\begin{proof}
 (1) We have $\red_{\C^{[n]}}(\A^n) \leq 2$ for all $n \gg 0$. By a result of Hoa, \cite[2.1]{Hoa}, the result follows.

 (2) Set $G = G_\A(A)$ and $L = L^\A(A)$.
 The first fundamental exact sequence $0 \rt G \rt L \rt L(-1) \rt 0$ yields an exact sequence in cohomology
 \[
 H^2(L)_n \rt H^2(L)_{n-1} \rt H^3(G)_n.
 \]
As $a(G) < 0$ we get for $n \geq 0$, surjections $H^2(L)_{n} \rt H^2(L)_{n-1}$. As $H^2(L)_n = 0$ for $n \gg 0$ we get
$H^2(L)_n = 0$ for $n \geq -1$.

After passing through a general extension we may choose $x$, an $A$-superficial element
\wrt  \ $\A$ such that in the ring $B = A/(x)$ the ideal $\B = \A/(x)$ is asymptotically normal. Also notice $\dim B = 2$.
Set $\ov{L} = L^\B(B)$. By \cite[9.2]{Pu5} we get that $H^1(\ov{L})_n = 0$ for $n < 0$.
By \ref{longH} we have an exact sequence
\[
 H^1(\ov{L})_n \rt H^2(L)_{n-1} \rt H^2(L)_n
\]
By setting $n = -1$ we get $H^2(L)_{-2} = 0$. Iterating we get $H^2(L)_n = 0$ for all $n \leq -2$.
It follows that $H^2(L) = 0$.

(3)  As $G_{\A^n}(A)$ is \CM \ for all $n \gg 0$, the ring $G_\A(A)$ is generalized \CM; see \cite[7.8]{Pu5}.

(4) As $G_\A(A)$ is generalized \CM \ we get that $H^i(G)$ has finite length for $i = 0, 1, 2$.
We also have $H^0(L)$
and $H^1(L)$ have finite length (see \ref{zero-lc} and \ref{normal}).

The first fundamental exact sequence $0 \rt G \rt L \rt L(-1) \rt 0$ yields an exact sequence in cohomology
 \begin{align*}
  0 &\rt H^0(G) \rt H^0(L) \rt H^0(L)(-1) \\
  &\rt H^1(G) \rt H^1(L) \rt H^1(L)(-1) \\
  &\rt H^2(G) \rt H^2(L) = 0.
 \end{align*}
 Taking lengths, the result follows.

 (5) As $\A$ is asymptotically normal we have that $H^1(L)_n = 0$ for $n < 0$ (see \cite[9.2]{Pu5}). Also by \ref{zero-lc} we get $H^0(L)_n = 0$
 for $n < 0$. The result follows from the above exact sequence in cohomology.
 \end{proof}

\begin{remark}\label{apply-gor-approx}
 Till now we have not used our theory of complete intersection approximation. We do it now.
 We first complete $A$. Let $\A = (a_1,\ldots, a_n)$ (minimally). If $\mu(\A) = 3$ then  $a_1, a_2, a_3$ is an $A$-regular sequence. In this case we have $G_\A(A)$ is \CM \ and thus we have nothing to prove. So we assume  $\mu(\A) \geq 4$. We take a general extension $A' = A[X_1,\ldots, X_n]_{\m A[X_1,\ldots, X_n]}$. Let $y = \sum_{i=1}^{n}a_iX_i$.
 Then the ideal $J = \A'/(y)$ is asymptotically normal. Set $B = A'/(y)$. By a result of Huckaba and Huneke we get that there exists $c$ such that $G_{J^n}(B)$ is \CM \ for $n \geq c$.  We take a CI-approximation $(R,\n,\B,\psi)$ of $A$ \wrt \ $\A$ (note \emph{not of} $A'$). By \ref{high-red} we may assume $\red(\B, R) \geq c+2$. By construction we may assume  that $\B$ is generated by $b_1,\ldots,b_n$ and $\psi(b_i) = a_i$ for all $i$. Now set $R' = R[X_1,\ldots, X_n]_{\n R[X_1,\ldots, X_n]}$. By \ref{lying-above} we get that $A' = A\otimes_R R'$. We note that $(R',\n',\B',\psi')$ is a CI-approximation of $(A',\m',\A')$. Set $z = \sum_{i=1}^{n}b_iX_i$. Then note $\psi'(z) = y$.
Also note that $z$ is $R'$-superficial \wrt \ $\B'$, see \cite[2.6]{Ciu} . We now complete $R'$ \wrt \ $\n'$.
 Thus we may assume that our ring $(A,\m, \A)$ has a CI-approximation $[T,\tf,\q,\psi]$ such that
 \begin{enumerate}
 \item Both $A$ and $T$ are complete.
   \item reduction number of $\q$ is $\geq c + 2$.
   \item there exists $z \in \q$ which is $A$-superficial \wrt \ $\A$ such that \\ $G_{\A^n}(A/zA)$ is \CM \ for all $n \geq c$. Furthermore $z^*$ is $G_\q(T)$-regular.
   \item $A$ has a canonical module $\omega_A = \Hom_T(A, T)$.
   \item We note that as $\q A = \A A$, we get $\R_\q(A) = \R_\A(A)$.
   \item Let $\eF$ be the dual filtration on $\omega_A$ \wrt \ $\q$, i,e,, we have an isomorphism
   \[
   \R(\eF, \omega_A) \cong \Hom_{\R_q(T)}(\R_\A(A), \R_\q(T)).
   \]
   As $\q \omega_A = \A \omega_A$, it follows that $\eF$ is infact an $\A$-stable filtration on $\omega_A$.
   \end{enumerate}

   Let $\R_\q(T)$ be the extended Rees-algebra of $T$ \wrt \ $\q$.  Let $\N$ be the $*$-maximal ideal of $\R_\q(T)$.
\end{remark}
We first show
\begin{lemma}\label{gcm}
 (with hypotheses as in \ref{apply-gor-approx}.) We have
 \begin{enumerate}[\rm (1)]
  \item For any prime $P$ in $\R_\q(T)$ with $\htt P \leq 3$ we have $\R_\A(A)_P$ is \CM.
  \item
  $H^i_\N(\R_\A(A))$ has finite length for $0 \leq i \leq 3$.
 \end{enumerate}
\end{lemma}
\begin{proof}
 As $\R_\q(T)$ is a Gorenstein ring the assertion (2) follows from (1) by local duality.

 (1) We first consider the case $t^{-1} \notin P$. Then $(\R_\A(A))_P$ is a localization of
 $(R_\A(A))_{t^{-1}} = A[t,t^{-1}]$ which is \CM.

 Next consider the case when $t^{-1} \in P$. Set $Q = P/(t^{-1})$ a prime ideal of height $\leq 2$ in $G_\q(T)$. As
 $G_\A(A)$ is generalized \CM \ and since $G_\q(T)$ is Gorenstein, by local duality we have that $(G_\A(A))_Q$ is \CM.
 It follows that $(\R_\A(A))_P$ is \CM.
\end{proof}
As a consequence we get
\begin{corollary}\label{exact-loc}
(with hypotheses as in \ref{apply-gor-approx}.) We have
\begin{enumerate}[\rm (1)]
  \item $H^3_\N(\R_\A(A)) = 0$
  \item an exact sequence
\[
 0 \rt H^3_\N(G_\A(A)) \rt H^4_\N(\R_\A(A))(+1) \xrightarrow{t^{-1}} H^4_\N(\R_\A(A)) \rt 0.
\]
  \item $\Ext^1_{\R_\q(T)}(\R_\A(A),\R_\q(T)) = 0$.
\end{enumerate}
\end{corollary}
\begin{remark}
The first assertion above also follows from \cite[Theorem 1.1]{KM}. However our proof is considerably simpler.
\end{remark}
\begin{proof}[Proof of Theorem \ref{exact-loc}]
 The short exact sequence
 $$0 \rt \R_\A(+1)  \xrightarrow{t^{-1}} \R_\A \rt G_\A(A) \rt 0,$$
 yields an exact sequence
 of local cohomology modules
 \begin{align*}
  0 &\rt H^0_\N(\R_\A(A))(+1) \xrightarrow{t^{-1}} H^0_\N(\R_\A(A)) \rt  H^0_\N(G_\A(A)) \\
  &\rt H^1_\N(\R_\A(A))(+1) \xrightarrow{t^{-1}} H^1_\N(\R_\A(A)) \rt  H^1_\N(G_\A(A)) \\
   &\rt H^2_\N(\R_\A(A))(+1) \xrightarrow{t^{-1}} H^2_\N(\R_\A(A)) \xrightarrow{\pi}  H^2_\N(G_\A(A)) \\
   &\rt H^3_\N(\R_\A(A))(+1) \xrightarrow{t^{-1}} H^3_\N(\R_\A(A))
 \end{align*}
Let $K$ be the cokernel of $\pi$. We note that $H^i_\N(G_\A(A)) \cong H^i(G_\A(A))$. As all the module in the above exact sequence have finite length
we have
\[
 \ell(K) = \sum_{i = 0}^{2} (-1)^i\ell(H^i(G_\A(A))  = 0 \quad \text{by Lemma \ref{rachel-itoh}(3)}.
\]
Thus we have an inclusion
\[
 H^3_\N(\R_\A(A))(+1) \xrightarrow{t^{-1}} H^3_\N(\R_\A(A))
\]
As $\ell(H^3_\N(\R_\A(A))) $ is finite we have an isomorphism
$H^3_\N(\R_\A(A))(+1) \cong H^3_\N(\R_\A(A))$. Again as $\ell(H^3_\N(\R_\A(A))) $ is finite this implies
that $H^3_\N(\R_\A(A)) = 0$.  So (1), (2) follows. By local duality we have (3).
\end{proof}
The following result is a crucial ingredient in proving Lemma \ref{crucial-itoh}.
\begin{theorem}\label{ing}(with hypotheses as in \ref{apply-gor-approx}.)
 Let $E_G$ be the injective hull of $k$ considered as a
 $G_\q(T)$-module. Set $W = \Hom_{G_\q(T)}(H^3(G_\A(A)), E_G)$. Let $s = \red(\q, T)$. Let $\eF$ be the dual filtration on $\omega_A$. Set $\eH = \eF(s-2)$. Then
 \begin{enumerate}[\rm (1)]
  \item
  $G_\eH(\omega_A) \cong W[1]$
 \item
 $ \depth G_\eH(\omega_A) \geq 2$. Furthermore $z^*$ is $G_\eH(\omega_A)$-regular.
 \item
 $\eH_n = \omega_A$ for $n \leq 0$.
 \item
 Set $\ov{T} = T/zT$, $\ov{\q} = \q/(z)$ and $B = A/zA$. Let $\ov{\eF}$ be the quotient filtration of $\eF$ on $\omega_B = \omega_A/z\omega_A$.
 Let $\eG$ be the dual filtration on $\omega_B$ \wrt \ $\R_{\ov{\q}}(\ov{T})$. Then $\ov{\eF} = \eG$.
 \item
 $(\R(\ov{H}(1),\omega_B))^{<s-1>}$ is a \CM \ $\R_{\q^{s-1}}(\ov{T})$-module.
 \item
 $\R(\ov{H}^{<s-1>}, \omega_B)$ is a \CM \ $\R_{\q^{s-1}}(\ov{T})$-module.
\item
 $H^1(G(\ov{\eH},\omega_B)_n = 0$ for $n < 0$.
 \end{enumerate}
\end{theorem}
\begin{proof}
(1) We note that the $a$-invariant of $G_\q(T)$ is $s-3$. So the $a$-invariant of $\R_\q(T)$ is $s-2$.
Let $E_\R$ be the injective hull of $k$ considered as a $\R_\q(T)$-module.  Set $(-)^\vee = \Hom_{\R_q(T)}(-, E_\R)$.
Dualizing the exact sequence in Lemma \ref{exact-loc} we get a sequence of $\R_\q(T)$-modules
\[
 0 \rt H^4_\N(\R_\A(A))^\vee \xrightarrow{t^{-1}}  H^4_\N(\R_\A(A))^\vee(-1) \rt W \rt 0.
\]
As $R_\q(T)$ is Gorenstein then by local duality we have an exact sequence
\begin{align*}
  0 & \rt \Hom_{\R_q(T)}(\R_\A(A), \R_\q(T)(s-2)) \xrightarrow{t^{-1}} \Hom_{\R_q(T)}(\R_\A(A), \R_\q(T)(s-2))(-1) \\
   &  \rt W \rt 0
\end{align*}
It follows that $G(\eH, \omega) = W[1]$

(2)
We note that  $W \cong \Hom_{G_\q(T)}(G_\A(A), G_\q(T) (s-3))$. We have that $G_\q(T)$ is \CM \ of dimension three and $z^*$ is $G_\q(T)$-regular.  The result follows.

(3)
 Set $G = G_\A(A)$.
Let $V = H^3(G)$. By \ref{rachel-itoh}(1) $a(G) < 0$. So  $V_n = 0$ for $n\geq 0$. Thus $W_m = 0$
for $m \leq 0$. Therefore $G(\eH, \omega_A)_n = 0$ for $n \leq -1$. The result follows.

(4) We have an exact sequence
\[
0 \rt \R_\q(T)(-1) \xrightarrow{zt} \R_\q(T) \rt   \R_\q(\ov{T}) \rt 0.
\]
This yields an exact sequence
\begin{align*}
  0 \rt \Hom_{\R_\q(T)}(\R_\A(A), \R_\q(T))(-1) &\xrightarrow{zt} \Hom_{\R_\q(T)}(\R_\A(A), \R_\q(T)) \rt  \\
  \Hom_{\R_\q(T)}(\R_\A(A), \R_\q(\ov{T})) &\rt \Ext^1_{\R_\q(T)}(\R_\A(A), \R_\q(T))(-1) \rt.
\end{align*}
By \ref{exact-loc} we get that $\Ext^1_{\R_\q(T)}(\R_\A(A), \R_\q(T)) = 0$.
We note that
\[
\Hom_{\R_\q(T)}(\R_\A(A), \R_\q(\ov{T})) = \Hom_{\R_\q(T)}(\R_\A(A)/zt \R_\A(A), \R_\q(\ov{T})).
\]
Set  $\B = \A/z A$. We have an exact sequence
\[
0 \rt F = \bigoplus_{n \geq 1}\frac{\A^n \cap (z) }{ z\A^{n-1}} \rt \R_\A(A)/zt \R_\A(A) \rt \R_\B(B) \rt 0.
\]
As $z$ is $A$-superficial \wrt \ $\A$  we get that $F$ has finite length. As $\R_\q(\ov{T})$ is \CM \  and has dimension $3$ we get
\[
\Hom_{\R_\q(T)}(\R_\A(A)/zt \R_\A(A), \R_\q(\ov{T})) \cong \Hom_{\R_\q(T)}(\R_\B(B), \R_\q(\ov{T})).
\]
Thus we have an exact sequence
\[
0 \rt \R(\eF, \omega_A)(-1)\xrightarrow{zt} \R(\eF,\omega_A) \rt \R(\eG,\omega_B) \rt 0;
\]
where $\eG$ is the dual filtration on $\omega_B$.
As $z^*$ is $G_\eF(\omega_A)$-regular it follows that $\eG$ is the quotient filtration of $\eF$ on $\omega_B$.

(5) By our construction $\R_{\B^n}(B)$ is \CM \ for $n \geq c$. Also $s = \red(\q, T) = \red(\ov{\q}, \ov{T}) \geq c + 2$.
As $G_{\ov{\q}}(\ov{T})$ is Gorenstein we get by \cite[Theorem 1]{Ooishi-gor-powers}  that $G_{\ov{\q}^{s-1}}(\ov{T})$ is Gorenstein. It follows that $Q = \R_{\ov{\q}^{s-1}}(\ov{T})$ is Gorenstein.
Using \ref{dual-Ver-prop} it follows that $\R(\ov{\eF}^{<s-1>},\omega_B)$ is a \CM \ $Q$-module. We note that $\R(\ov{\eF}^{<s-1>},\omega_B) = \R(\ov{\eF}, \omega_B)^{<s-1>}$.
Observe
\[
 \R(\ov{\eF}(s-1), \omega_B)^{<s-1>} =  \left(\R(\ov{\eF}, \omega_B)^{<s-1>}\right)(s-1).
\]
It follows that $ \R(\ov{\eF}(s-1), \omega_B)^{<s-1>}$ is a \CM \ $Q$-module.
So \\ $\R(\ov{\eH}(1), \omega_B)^{<s-1>}$ is a \CM \ $Q$-module.

(6) We have $\R(\ov{\eH}, \omega_B) = \bigoplus_{n \in \Z} \ov{\eH}_n$. So we get
\[
\R(\ov{\eH}(1), \omega_B)^{<s-1>} = \bigoplus_{n \in \Z} \ov{\eH}_{(n+1)(s-1)}.
\]
We also have $\R(\ov{\eH}, \omega_B)^{<s-1>} = \bigoplus_{n \in \Z} \ov{\eH}_{n(s-1)}$. Thus we have
\[
\left(\R(\ov{\eH}, \omega_B)^{<s-1>}\right)_n = \ov{\eH}_{n(s-1)} = \left(\R(\ov{\eH}(1), \omega_B)^{<s-1>}\right)_{n-1}.
\]
So we have an equality of $Q = \R_{\ov{\q^{s-1}}}(\ov{T})$-modules
\[
\left(\R(\ov{\eH}, \omega_B)^{<s-1>}\right) \cong \left(\R(\ov{\eH}(1), \omega_B)^{<s-1>}\right)(-1).
\]
Thus $\R(\ov{\eH}, \omega_B)^{<s-1>}$ is a \CM \ $Q$-module.  We also note  that \\ $\R(\ov{\eH}, \omega_B)^{<s-1>} = \R(\ov{\eH}^{<s-1>}, \omega_B)$.

(7) By (6) we get that $G(\ov{\eH}^{<s-1>}, \omega_B)$ is \CM. The result now follows from \ref{asympCM}.
\end{proof}

We now give
\begin{proof}[Proof of Lemma \ref{crucial-itoh}]
By  Lemma \ref{rachel-itoh}(2) we get $H^2(L^\A(A)) = 0$. By \ref{dag} we have an exact sequence
\[
H^1(L^\A(A)) \rt H^1(L^\A(A))(-1) \rt H^2(G_\A(A)) \rt H^2(L^\A(A)) = 0
\]
As $H^1(L^\A(A))_n = 0$ for $n < 0$ it follows that $H^2(G_\A(A))_n = 0$ for $n \leq 0$. Note that to prove $ H^1(L^\A(A)) = 0$ it suffices to show $H^2(G_\A(A)) = 0$ (because
$H^1(L^\A(A))$ has finite length).
Set $G = G_\A(A)$ and $\ov{G} = G/ztG$. We have a exact sequence
\[
0 \rt K \rt G(-1) \xrightarrow{zt} G \rt \ov{G} \rt 0;
\]
where $K$ has finite length. Taking local cohomology we get an exact sequence
\begin{align*}
  H^2(G)(-1)&\xrightarrow{zt} H^2(G) \rt H^2(\ov{G}) \\
  H^3(G)(-1)&\xrightarrow{zt} H^3(G) \rt 0.
\end{align*}
Set $Y = H^2(G)/ztH^2(G)$. Note $Y_n = 0$ for $n \leq 0$. Taking Matlis dual's  we get an exact sequence
\[
0 \rt H^3(G)^\vee \xrightarrow{zt} H^3(G)^\vee(1) \rt H^2(\ov{G})^\vee \rt Y^\vee \rt 0.
\]
By \ref{ing} we get an exact sequence
\[
0 \rt G(\ov{\eH}, \omega_B)  \rt H^2(\ov{G})^\vee \rt Y^\vee \rt 0.
\]
We note that $\depth H^2(\ov{G})^\vee \geq 2$ and $H^1(G(\ov{\eH}, \omega_B))_n = 0$ for $n < 0$. So $Y^\vee_n = 0$ for $n < 0$.
It follows that $Y_n = 0$ for $n > 0$. As $Y_n = 0$ for $n \leq 0$ we get $Y = 0$. So $H^2(G)/zt H^2(G) = 0$. By graded Nakayama Lemma we get $H^2(G) = 0$. As discussed before this implies that $H^1(L^\A(A)) = 0$.
\end{proof}

\emph{Acknowledgments:}
I thank  Lucho Avramov, Juergen Herzog, Srikanth Iyengar  and Jugal Verma for
 many useful discussions on the subject of this paper.

\providecommand{\bysame}{\leavevmode\hbox to3em{\hrulefill}\thinspace}
\providecommand{\MR}{\relax\ifhmode\unskip\space\fi MR }
\providecommand{\MRhref}[2]{%
  \href{http://www.ams.org/mathscinet-getitem?mr=#1}{#2}
}
\providecommand{\href}[2]{#2}

\end{document}